%% file: main.tex
\documentclass[11pt, a4paper]{amsart}
\title{A distributive lattice of model structures\\ relating to spectral sequences}
\author{James A. Brotherston}
\address{School of Mathematics and Statistics, University of Sheffield, S3 7RH, UK}
\email{brotherston.maths@gmail.com}
\date{}
\thanks{This work was supported by the Engineering and Physical Sciences Research Council.}
\input{./tex/packages.tex}
\input{./tex/theorems.tex}

\input{./tex/commands.tex}

\begin{document}
\maketitle
\begin{abstract}
  \input{./tex/abstract.tex}
\end{abstract}
\input{./tex/introduction.tex}
\input{./tex/acknowledgements.tex}
\input{./tex/preliminaries.tex}
\input{./tex/adjoints.tex}
\input{./tex/properness.tex}
\input{./tex/cellularity.tex}
\input{./tex/stability.tex}
\input{./tex/bousfieldLocalisations.tex}
\input{./tex/quillenEquivalence.tex}
\appendix

\bibliographystyle{alpha}
\bibliography{./tex/bibliography}
\end{document}

%% file: tex/packages.tex
\usepackage{fullpage}
\usepackage{tikz}
\usepackage{tikz-cd}
\usetikzlibrary{calc}
\usepackage{pict2e}
\usepackage{hyperref}
\hypersetup{pdftitle={A distributive lattice of model structures relating to spectral sequences},pdfauthor={James A. Brotherston}}
\usepackage{amsmath}
\usepackage{amssymb}
\usepackage{mathdots} 
\usepackage{mathtools}
\usepackage{amsthm}
\usepackage[noabbrev,capitalise]{cleveref}
\usepackage{todonotes}

%% file: tex/theorems.tex
\theoremstyle{plain}
\newtheorem{theo}{Theorem}[subsection]
\Crefname{theo}{Theorem}{Theorems}
\newtheorem{lemm}[theo]{Lemma}
\Crefname{lemm}{Lemma}{Lemmas}
\newtheorem{prop}[theo]{Proposition}
\Crefname{prop}{Proposition}{Propositions}
\newtheorem{coro}[theo]{Corollary}
\Crefname{coro}{Corollary}{Corollaries}
\theoremstyle{definition}
\newtheorem{defi}[theo]{Definition}
\Crefname{defi}{Definition}{Definitions}

\Crefname{exam}{Example}{Examples}
\theoremstyle{remark}

\Crefname{rema}{Remark}{Remarks}

\Crefname{nota}{Notation}{Notations}

\theoremstyle{plain}
\newtheorem{theoL}{Theorem}
\Crefname{theoL}{Theorem}{Theorems}

\Crefname{lemmL}{Lemma}{Lemmas}

\newtheorem{propL}[theoL]{Proposition}
\Crefname{propL}{Proposition}{Propositions}

\Crefname{coroL}{Corollary}{Corollaries}

%% file: tex/commands.tex
\newcommand{\mc}{\mathcal}
\newcommand{\Zbb}{\mathbb{Z}}

\let\amsamp=&
\renewcommand{\phi}{\varphi}
\newcommand{\new}{\mathrm{new}}

\DeclareMathOperator*{\im}{im}
\newcommand{\id}{\mathrm{id}}

\newcommand{\Hom}{{\mathrm{Hom}}}

\newcommand{\mathdash}{\text{-}}
\newcommand{\Fib}{\mathrm{Fib}}

\newcommand{\Inj}{\mathrm{Inj}}
\newcommand{\Proj}{\mathrm{Proj}}
\newcommand{\Cell}{\mathrm{Cell}}
\newcommand{\Cof}{\mathrm{Cof}}
\newcommand{\Quill}{\simeq_Q}
\newcommand{\cellularization}{\mathrm\mathrm{cell}\mathrm}
\newcommand{\fCh}{f\mathcal{C}}

\newcommand{\fChS}{\left(\fCh\right)_S}
\newcommand{\fChT}{\left(\fCh\right)_T}
\newcommand{\Er}{\mathcal{E}_r}
\newcommand{\Shift}{S}
\newcommand{\Dec}{\mathrm{Dec}}
\newcommand{\bCh}{b\mathcal{C}}

\newcommand{\bChS}{\left(\bCh\right)_S}
\newcommand{\Z}{\mathcal{Z}}
\newcommand{\B}{\mathcal{B}}
\newcommand{\ZW}{\mathcal{ZW}}
\newcommand{\BW}{\mathcal{BW}}
\newcommand{\NW}{\mathcal{NW}}

\newcommand{\Totp}{\mathrm{Tot}^{\Pi}}
\newcommand{\Totc}{\mathrm{Tot}^{\oplus}}
\newcommand{\rSusp}{\Sigma^r}
\newcommand{\rLoops}{\Omega^r}
\DeclareMathOperator{\map}{\mathrm{map}}
\makeatletter
\newcommand{\noloc}{\nobreak\mskip6muplus1mu{:}\nonscript
  \mkern-\thinmuskip\mathpunct{}\mskip2mu\relax}
\newcommand{\smallbot}{%
  \begingroup\setlength\unitlength{.15em}%
  \begin{picture}(1,1)
    \roundcap
    \polyline(0,0)(1,0)
    \polyline(0.5,0)(0.5,1)
  \end{picture}%
  \endgroup
}
\newcommand{\inadj}[4]{
  #1\colon #2%
  \mathrel{\vcenter{%
      \offinterlineskip\m@th
      \ialign{%
        \hfil$##$\hfil\cr
        \longrightarrow\cr
        \noalign{\kern-.3ex}
        \smallbot\cr
        \longleftarrow\cr
      }%
    }}%
  #3 \noloc #4%
}
\newcommand{\inadjarrows}[2]{
  #1%
  \mathrel{\vcenter{%
      \offinterlineskip\m@th
      \ialign{%
        \hfil$##$\hfil\cr
        \longrightarrow\cr
        \noalign{\kern-.3ex}
        \smallbot\cr
        \longleftarrow\cr
      }%
    }}%
  #2%
}
\makeatother
\newcommand{\adj}[4]{
  \begin{tikzcd}[ampersand replacement=\&]
    {#1} \colon {#2} \arrow[r, shift left=.5ex]
    \arrow[r, phantom, "\smallbot"] \& {#3} \noloc {#4}
    \arrow[l, shift left=.5ex]
  \end{tikzcd}%
}

\newlength\stextwidth


%% file: tex/abstract.tex
The $S$-model category structures on filtered chain complexes and bicomplexes were introduced by Cirici, Egas Santander, Livernet and Whitehouse and later generalised by this author.
In this paper we show they are left proper, cellular and stable model categories.
We use these properties and the Cellularization Principle of Greenlees and Shipley to show that an adjunction with right adjoint the product totalisation functor from bicomplexes to filtered chains is a Quillen equivalence.
Combined with other known Quillen equivalences between filtered chains this shows these model categories all present the same homotopy category.
We also construct a distributive lattice whose elements are the $S$-model categories of filtered chain complexes.


%% file: tex/introduction.tex
\section{Introduction}\label{introduction}
In \cite{CELW} the authors define, for each $r\geq 0$, model category structures on the categories of filtered chain complexes $\fCh$  and bicomplexes $\bCh$ whose weak equivalences are the quasi-isomorphisms between the $r$-pages of the associated spectral sequences.
For each $r$ these came in two flavours. For $\fCh$, one whose fibrations are bidegreewise surjective on the \textit{$r$-cycles} and the other whose fibrations are bidegreewise surjective on all $k$-cycles for $0\leq k\leq r$.
For $\bCh$ the analogous fibrations are defined with \textit{witness cycles} and additionally require $0$-witness cycle surjectivity in both cases.

The interest in such model structures is in situations where one has constructions well defined up to an isomorphism on some page of an associated spectral sequence.
Some interesting examples are listed in the introduction of \cite{CELW} including examples in \textit{mixed Hodge theory} and \textit{rational homotopy theory}.

In \cite{BMonoidal} this author generalised their model categories to include those indexed by any non-empty finite subset $S$ of $\mathbb{N}$ whose fibrations are bidegreewise surjective on $s$-(witness) cycles for all $s\in S$.
We denote these $\fChS$ and $\bChS$.
Our convention is that $0\in\mathbb{N}$ and when $S$ indexes $\bChS$ that $0$ is in $S$.
We generally let $r$ denote the maximum of $S$ and this defines the weak equivalences as above.
For $\fChS$ we also demonstrated these are monoidal model structures as in \cite{H}, 
and showed the existence of various model categories of modules and algebras as follow from Schwede and Shipley's monoid axiom.

This paper further investigates the model categories $\fChS$ and $\bChS$.
Our main result is the following which is shown in \cref{allModCatsQuillenEquiv}.
\begin{theoL}
  There is a zig-zag of Quillen equivalences between any two of the model categories $\fChS$ and $\left(\bCh\right)_T$ where $S$ and $T$ are any non-empty finite subsets of $\mathbb{N}$ with $0\in T$.
  The underlying adjunctions of the Quillen equivalences in question are either of the form identity-identity, shift-d\'ecalage or a product totalisation adjunction.

  In particular then all their homotopy categories are equivalent.
\end{theoL}

The identity-identity and shift-d\'ecalage Quillen equivalences are generalisations of those shown in \cite{CELW}.
For the product totalisation Quillen equivalences we first construct a left adjoint to the totalisation functor $\Totp$ which we denote $\mc{L}$.
For $S\subset\mathbb{N}\cup\left\{0\right\}$ finite, non-empty and including $0$, there is immediately a Quillen adjunction $\inadj{\mc{L}}{\fChS}{\bChS}{\Totp}$ by virtue of how weak equivalences and fibrations are defined in $\bChS$.
Using the Cellularization Principle of Greenlees and Shipley, \cite{GS}, we are able to show the following in \cref{LTotEquivalence}.

\begin{theoL}\label{LTotQEquivIntro}
  Let $S\subseteq\left\{0,1,2,\ldots,r\right\}$ containing both $0$ and $r$.
  There is a Quillen equivalence $\inadj{\mc{L}}{\fChS}{\bChS}{\Totp}$.
\end{theoL}

Our application of the Cellularization Principle to prove \cref{LTotQEquivIntro} requires a study of $\mc{L}$ applied to representing objects for the $s$-cycles, right properness of the model categories involved and the following new properties given by \cref{propertiesIntro}.

\begin{propL}\label{propertiesIntro}
  Each of the model categories $\fChS$ and $\bChS$ is left proper, cellular and stable.
\end{propL}

Right properness follows by construction as all objects of the $\fChS$ and $\bChS$ are fibrant.
Cellularity is shown in \cref{cellularity} and stability in \cref{stability} by explicitly computing the homotopy pullback using a fibration from the \textit{$r$-cones} in $\fCh$ and $\bCh$.
The loops and suspension functors applied to $A\in\fCh$ can be modelled by a tensor product by a copy of the ground ring $R$ concentrated in some filtration and homological degree and for bicomplexes a tensor product by $R$ concentrated in some bidegree.
The effect of the loops and suspension functors is a shift by the bidegree of the $r$-page differential.

We make use of a result of Dugger, \cite{Dugger}, to identify a set of cofibrant objects detecting weak equivalences.
This requires knowing $\fChS$ and $\bChS$ are left proper which we demonstrate using a result of Lack, \cite{Lack}.

For any $r\geq0$ the $(r+1)$-quasi isomorphisms are a subclass of the $r$-quasi isomorphisms. Despite this in \cref{bousfieldLocalisations} we show there is no left Bousfield localisation from an $S$-model structure, on either $\fCh$ or $\bCh$ with $\max S=r$, to a model structure with weak equivalences being the $(r+1)$-quasi isomorphisms.

We finish by describing a distributive lattice structure on the set of finite non-empty subsets of $\mathbb{N}$, i.e.\ the set of indexing elements of the model categories $\fChS$.
We denote this distributive lattice by $\mc{N}$.
The poset structure is given by defining $T\leq S$ if there is a composite of left Quillen adjoints of the identity-identity and shift-d\'ecalage adjunctions from $\fChS$ to $\left(\fCh\right)_T$.
In terms of the elements of $\mc{N}$ the inequality is generated by the assertion that $T<S$ if either:
\begin{itemize}
\item $T\subset S$ and $\max T=\max S$, or
\item $S=T+1\coloneqq \left\{t+1\,|\, t\in T\right\}$.
\end{itemize}
Our meet and join operations on $S,T\in\mc{N}$ are informally described as ``the terminal model structure admitting left adjoints to both $\fChS$ and $\fChT$'' and ``the initial model structure admitting left adjoint from both $\fChS$ and $\fChT$'' in $\mc{N}$.
Explicit formulae for these operations on $S$ and $T$ are given in \cref{MeetJoinFormulae}.

By exhibiting an isomorphism from this poset with these meet and join operations to another distributive lattice using Birkhoff's representation theorem, \cref{Birkhoff}, we show the following theorem.

\begin{theoL}
  $\mc{N}$ with meet and join operations as above is a distributive lattice.
\end{theoL}

This paper is based on work from the author's Ph.D.\ thesis, \cite{BThesis}.


%% file: tex/acknowledgements.tex
\section*{Acknowledgements}

I would like to express my deep gratitude to my Ph.D.\ supervisor Sarah Whitehouse whose support and guidance this research was carried out under.
I'm also very grateful to Daniel Graves, Luca Pol and Jordan Williamson for many helpful conversations.

%% file: tex/preliminaries.tex
\section{Preliminaries}\label{preliminiaries}
Throughout $R$ will denote a fixed commutative unital ring.
Chain complexes are graded cohomologically and filtrations are increasing filtrations.
Bicomplexes are graded cohomologically vertically and homologically horizontally.

We recall the necessary preliminaries on filtered chain complexes, bicomplexes, the shift-d\'ecalage adjunction, spectral sequences, model categories, Bousfield localisations, the Cellularization Principle, and construction of the $S$-model structures on filtered chain complexes and bicomplexes.

Our conventions agree with those of \cite{CELW} minus a choice of sign on certain differentials.
\subsection{Filtered chain complexes}
\begin{defi}\label{filtChainsObjsMorphs}
  A \textit{filtered chain complex} $A$ is a chain complex equipped with an \textit{increasing filtration}, i.e.\ subcomplexes $F_pA$ with $F_pA\subseteq F_{p+1}A\subseteq A$ for each $p\in\Zbb$. A \textit{morphism of filtered chain complexes} $f\colon A\rightarrow B$ is a morphism of the underlying chain complexes which preserves the filtration, i.e.\ $f(F_pA)\subseteq F_pB$.
\end{defi}
\begin{defi}\label{filtChainsCat}
  The \textit{category of filtered chain complexes} with objects and morphisms given as in \cref{filtChainsObjsMorphs} will be denoted $\fCh$.
\end{defi}

\begin{defi}
  An $A\in\fCh$ is said to be of pure filtration degree $p$ if $0=F_{p-1}A\subset F_pA=A$.
  We denote by $R_{(p)}^n$ a graded filtered $R$-module given by a copy of $R$ concentrated in pure filtration degree $p$ and concentrated in cohomological degree $n$.
\end{defi}
We use this notation to build further objects of $\fCh$ later, e.g.\ $R_{(p)}^n\rightarrow R_{(p-r)}^{n+1}$ will denote a direct sum of two such objects with an identity differential appearing in filtration degrees $p$ and above.

We will need to have an understanding of the small and finite objects in filtered chain complexes in dealing with colimits, as in the likes of the small objects argument.
We take \cite[Definitions 2.1.3 and 2.1.4]{H} as our definitions of small and finite objects respectively.
The following two lemmas are proved in a similar way to \cite[Lemma 2.3.2]{H}.
\begin{lemm}\label{smallFilteredChains}
  Every object of $\fCh$ is small relative to the whole category.\qed
\end{lemm}
\begin{lemm}\label{FiniteFilteredChains}
  A filtered chain complex $A$ is a finite object of the category $\fCh$ if and only if it satisfies the following conditions:
  \begin{enumerate}
  \item $F_pA^n$ is finitely presented for all $p$ and $n$,
  \item $A^n=0$ for all $n\leq n_1$ for some $n_1$,
  \item $A^n=0$ for all $n\geq n_2$ for some $n_2$,
  \item $F_{p_1}A=0$ for some finite $p_1$, and
  \item $F_{p_2}A=A$ for some finite $p_2$.\qed
  \end{enumerate}
\end{lemm}

Recall on the category of chain complexes there are functors $\Sigma$ and $\Omega$ which are inverse to each other given by $(\Sigma X)^n\coloneqq X^{n+1}$ and $d_{\Sigma X}^n=-d_X^{n+1}$.
In \cite[Definition 3.5]{CELW} the authors refer to the following as the \textit{$r$-translation} of a filtered chain complex.
\begin{defi}\label{rSuspensionAndLoopsFilteredChains}
  For an $A\in\fCh$ the \textit{$r$-suspension} of $A$ denoted $\rSusp A$ is the filtered chain complex with underlying chain complex that of $\Sigma A$ and filtration given by $F_p(\rSusp A)\coloneqq F_{p-r}A^{n+1}$.
  The \textit{$r$-loops} of $A$ denoted $\Omega^rA$ is the inverse of $\rSusp$.
\end{defi}

\subsection{Shift-d\'ecalage adjunction}
The shift and d\'ecalage functors are endofunctors of $\fCh$ exhibiting an adjunction as shown by Deligne in \cite{Deligne}.
\begin{defi}\label{shiftDecalage}
  Let $r\geq 0$ and $A\in\fCh$. The \textit{shift} and \textit{d\'ecalage} endofunctors $\Shift^r$ and $\Dec^r$ on $\fCh$ are the identity on the underlying chain complex, modify the filtration by:
  \begin{align*}
    F_p(\Shift^rA)^n&\coloneqq F_{p+rn}A^n,\\
    F_p(\Dec^rA)^n&\coloneqq Z_r^{p-rn,p-rn+n}(A),
  \end{align*}
  and are such that $\Shift^r$ is equal to r composites of $\Shift\coloneqq\Shift^1$ and $\Dec^r$ is r composites of $\Dec\coloneqq\Dec^1$.
\end{defi}
\begin{lemm}[{\cite[\S2.3]{CiriciGuillen}}]\label{shiftDecalageAdjunction}
  For each $r\geq 0$ there is an adjunction $\Shift^r\dashv\Dec^r$ for which the unit $\id \Rightarrow \Dec^r\circ\Shift^r$ is the identity natural transformation.\qed
\end{lemm}
\subsection{Bicomplexes}
\begin{defi}\label{bicomplex}
  A \textit{bicomplex} $A$ over $R$ is a bigraded collection of $R$-modules $A^{i,j}$ for $i,j\in\Zbb$ and differentials $d_0\colon A^{i,j}\rightarrow A^{i,j+1}$ and $d_1\colon A^{i,j}\rightarrow A^{i-1,j}$ which square to zero $d_0d_0=0$, $d_1d_1=0$ and are required to commute $d_0d_1=d_1d_0$. 
\end{defi}
\begin{defi}\label{bicompCat}
  The \textit{category of bicomplexes} with objects and morphisms given as in \cref{bicomplex} will be denoted $\fCh$.
\end{defi}
\begin{defi}
  The \textit{product totalisation functor} $\Totp\colon\bCh\rightarrow\fCh$ is defined on a bicomplex $A$ by:
  \begin{equation*}
    \Totp(A)^n\coloneqq \prod_{i\in\Zbb}A^{i,i+n}
  \end{equation*}
  with filtration given by:
  \begin{equation*}
    F_p\Totp(A)\coloneqq \prod_{i\leq p}A^{i,i+n}
  \end{equation*}
  and differential $d^{\Totp}$ on an element $(a_i)_i\in\Totp(A)^n$ given by:
  \begin{equation*}
    d^{\Totp}\colon(a_i)_i\mapsto(d_0a_i+(-1)^nd_1a_{i+1})_i\;.
  \end{equation*}
  On a morphism $f\colon A\rightarrow B$ we have $\Totp(f)^n\coloneqq \prod_{i\in\Zbb}f^{i,i+n}$.
\end{defi}
Similarly there is a \textit{coproduct totalisation functor} denoted $\Totc$ obtained by replacing the product $\prod$ by the coproduct $\bigoplus$ in the preceding definition.
\begin{defi}
  Let $A\in\bCh$. We define the \textit{$r$-suspension} $\rSusp A$ and \textit{$r$-loops} $\rLoops A$ to be the bicomplexes with:
  \begin{align*}
    \left(\rSusp A\right)^{p,q}&\coloneqq A^{p-r,q-r+1}\\
    \left(\rLoops A\right)^{p,q}&\coloneqq A^{p+r,q+r-1}
  \end{align*}
  and whose differentials are $d_0^{\rLoops A}=d_0^{\rSusp A}=(-1)^{r+1}d_0^A$ and $d_1^{\rLoops A}=d_1^{\rSusp A}=(-1)^{r}d_0^A$.
\end{defi}
The following lemma is proved in the same way as \cite[Lemma 2.3.2]{H}.
\begin{lemm}\label{smallBicomplexes}
  Every object of $\bCh$ is small relative to the whole category.
  The finite objects are those objects that are bounded from above, below and from the right and left which are bidegreewise finitely generated $R$-modules. \qed
\end{lemm}

\subsection{Spectral sequences and $r$-cones}
\begin{defi}
  A \textit{spectral sequence} is a sequence of $(\Zbb,\Zbb)$-bigraded differential modules $\left\{E_r^{\bullet,\bullet}, d_r\right\}$ for $r\geq 0$ such that the differentials $d_r$ are of bidegree $(-r,1-r)$, $d_r\colon E_r^{p,q}\rightarrow E_r^{p-r,q+r-1}$, and further that:
  \begin{equation*}
    E_{r+1}^{p,q}\cong H^{p,q}(E_{r}^{\bullet,\bullet},d_R)
    =\frac{\ker\left(d_r\colon E_r^{p,q}\rightarrow E_r^{p-r,q+1-r}\right)}
    {\im\left(d_r\colon E_r^{p+r,q+1-r}\rightarrow E_r^{p,q}\right)}
  \end{equation*}
  for each $r\geq 0$. We refer to the $E_r^{\bullet,\bullet}$ as the \textit{$r$-page} of the spectral sequence.
\end{defi}
In \cite{CELW} the authors describe a setup for spectral sequences of filtered chain complexes and bicomplexes such that the cycle and boundary functors are representable.
We recall their definitions here.

\begin{defi}\label{associatedSS}
  Let $A$ be a filtered chain complex. We define in bidegree $(p,p+n)$:
  \begin{itemize}
  \item for $r\geq 0$ the \textit{$r$-cycles} of $A$ as:
    \begin{equation*}
      Z_r^{p,p+n}(A)\coloneqq F_pA^n\cap d^{-1}F_{p-r}A^{n+1}\,,
    \end{equation*}
  \item for $r=0$ the \textit{$r$-boundaries} of $A$ as:
    \begin{equation*}
      B_0^{p,p+n}(A)\coloneqq Z_0^{p-1,p-1+n}(A) = F_{p-1}A^n\,,
    \end{equation*}
  \item and for $r\geq 1$ the \textit{$r$-boundaries} of $A$ in bidegree $(p,p+n)$ as:
    \begin{equation*}
      B_r^{p,p+n}(A)\coloneqq dZ_{r-1}^{p+r-1,p+r-1+n+1}(A)+Z_{r-1}^{p-1,p-1+n}(A)\,.
    \end{equation*}
  \end{itemize}
  The functors $Z_r^{p,p+n}$ and $B_r^{p,p+n}$ are representable in $\fCh$ with representing objects:
  \begin{align*}
    \Z_r(p,n)&\coloneqq \left(R_{(p)}^n\rightarrow R_{(p-r)}^{n+1}\right)\\
    \B_r(p,n)&\coloneqq \left(R_{(p+r-1)}^{n-1}\rightarrow R_{(p)}^n\right)\oplus
               \left(R_{(p-1)}^n\rightarrow R_{(p-r)}^{n+1}\right)
  \end{align*}
  and there is a morphism $\phi_r\colon\Z_r(p,n)\rightarrow\B_r(p,n)$ which is the diagonal in degree $n$ and the identity in degree $n+1$ whenever possible.
  This then induces a map $w_{r}\colon B_r^{p,p+n}(A)\rightarrow Z_r^{p,p+n}(A)$.
  The \textit{$r$-page of the associated spectral sequence} denoted $E_r^{p,p+n}(A)$ is then given by:
  \begin{equation*}
    E_r^{p,p+n}(A)\coloneqq \frac{Z_r^{p,p+n}(A)}{B_r^{p,p+n}(A)}
    =\frac{\Hom_{\fCh}(\Z_r(p,n),A)}{\Hom_{\fCh}(\B_r(p,n),A)}\;.
  \end{equation*}
  The differential $d$ restricts to $r$-cycles and there is then an induced differential on the $r$-page sending a class $[a]$ to $[da]$.
\end{defi}

\begin{defi}
  Given a morphism $f\colon A\rightarrow B$ of $\fCh$ we will say $f$ is \textit{$Z_k$-bidegreewise surjective (resp.\ injective)} if $Z_k^{p,p+n}(f)\colon Z_k^{p,p+n}(A)\rightarrow Z_k^{p,p+n}(B)$ is surjective (resp.\ injective) for all $p,n\in\Zbb$ and similarly with the functors $B_k^{\ast,\ast}$ and $E_k^{\ast,\ast}$ as well.
\end{defi}

\begin{defi}
  The class of morphisms of filtered chain complexes inducing an isomorphism between the $(r+1)$-pages of the associated spectral sequences is denoted $\Er$.
  These will be called the \textit{$r$-quasi-isomorphisms} or \textit{$r$-weak equivalences}.
\end{defi}

One can immediately obtain spectral sequences given a bicomplex $A$ via the $\Totp$ functor and taking the above cycles and boundaries of $\Totp(A)$.
However we need representing objects in $\bCh$ so the authors of \cite{CELW} instead define \textit{$r$-witness cycles} and \textit{$r$-witness boundaries} whose quotients yield the same spectral sequence.

\begin{defi}
  Let $A$ be a bicomplex. We define in bidegree $(p,p+n)$:
  \begin{itemize}
  \item for $r=0$ the \textit{$r$-witness cycles} of $A$ as:
    \begin{equation*}
      ZW_0^{p,p+n}(A)\coloneqq A^{p,p+n},
    \end{equation*}
  \item for $r>0$ the \textit{$r$-witness cycles} of $A$ as:
    \begin{align*}
      ZW_r^{p,p+n}(A)\coloneqq \left\{(a_0,a_1,\ldots,a_{r-1})\in
        \left.\bigoplus_{k=0}^{r-1}A^{p-k,p+n-k}\,\right|\,\right.
      &d_0a_0=0,\\
      &\left.d_0a_k=d_1a_{k-1} \text{ for } k\geq 1
        \vphantom{\bigoplus_{k=0}^{r-1}}\right\}\,,
    \end{align*}
  \item for $r=0$ the \textit{$r$-witness boundaries} of $A$ as:
    \begin{equation*}
      BW_0^{p,p+n-1}(A)\coloneqq 0,
    \end{equation*}
  \item for $r=1$ the \textit{$r$-witness boundaries} of $A$ as:
    \begin{equation*}
      BW_1^{p,p+n-1}(A)\coloneqq A^{p,p+n-1},
    \end{equation*}
  \item for $r\geq 2$ the \textit{$r$-witness boundaries} of $A$ as:
    \begin{equation*}
      BW_r^{p,p+n-1}(A)\coloneqq ZW_{r-1}^{p+r-1,p+r+n-2}(A)\oplus A^{p,p+n-1}\oplus
      ZW_{r-1}^{p-1,p-1+n}(A)\,.
    \end{equation*}
  \end{itemize}
  The functors $ZW_r^{p,p+n}$ are representable with representing objects $\ZW_r(p,p+n)$ as indicated in \cref{figu:ZWr} with $\bullet$ labelling a copy of $R$, arrows indicating identity differentials and bidegrees of some components specified.
  Similarly the $BW_r^{p,p+n-1}$ are represented by $\BW_r(p,p+n-1)$ which are direct sums of representing witness cycles.
  \begin{align*}
    \BW_0(p,p+n-1)\coloneqq& 0\\
    \BW_1(p,p+n-1)\coloneqq& \ZW_0(p,p+n-1)\\
    \BW_r(p,p+n-1)\coloneqq&\ZW_{r-1}(p+r-1,p+r+n-2)\\
                         &\oplus\ZW_0(p,p+n-1)\oplus\ZW_{r-1}(p-1,p-1+n)
  \end{align*}
\begin{figure}[h]
  \centering
  \begin{tikzpicture}
    \node (a0) at (0,0) {};
    \filldraw (a0) circle (2pt);
    \node (a1) at (-1,-1) {};
    \filldraw (a1) circle (2pt);
    \node (a2) at (-2,-2) {};
    \filldraw (a2) circle (2pt);
    \node (a3) at (-3,-3) {};
    \filldraw (a3) circle (2pt);
    \node (a4) at (-4,-4) {};
    \filldraw (a4) circle (2pt);
    \node (b0) at (-1,0) {};
    \filldraw (b0) circle (2pt);
    \node (b1) at (-2,-1) {};
    \filldraw (b1) circle (2pt);
    \node (b3) at (-4,-3) {};
    \filldraw (b3) circle (2pt);
    \node (b4) at (-5,-4) {};
    \filldraw (b4) circle (2pt);
    \draw[->] (a0) node[anchor=north west] {$(p,p+n)$} -- (b0);
    \draw[->] (a1) -- (b1);
    \draw[->] (a3) -- (b3);
    \draw[->] (a4) -- (b4);
    \draw[->] (a1) -- (b0);
    \draw[->] (a2) -- (b1);
    \draw[->] (a4) node[anchor=north west] {$(p-r+1,p-r+1+n)$} -- (b3);
    \node (a) at (-2.5,-2.5) {$\iddots$};
    \node (b) at (-3,-2) {$\iddots$};
    \node (00) at (-7,-2) {};
    \filldraw (00) circle (2pt);
    \node (-10) at (-8,-2) {};
    \filldraw (-10) circle (2pt);
    \node (01) at (-7,-1) {};
    \filldraw (01) circle (2pt);
    \node (-11) at (-8,-1) {};
    \filldraw (-11) circle (2pt);
    \draw[->] (00) node[anchor=north west] {$(p,p+n)$} -- (-10);
    \draw[->] (00) -- (01);
    \draw[->] (01) -- (-11);
    \draw[->] (-10) -- (-11);
  \end{tikzpicture}
  \caption{The bicomplexes $\ZW_0(p,p+n)$ and $\ZW_r(p,p+n)$ for $r\geq 1$}
  \label{figu:ZWr}
\end{figure}

There are morphisms $\phi\colon\ZW_r(p,p+n)\rightarrow\BW_r(p,p+n-1)$ which are bidegreewise firstly the diagonal map wherever possible, secondly the identity map wherever possible, and otherwise $0$.
These were defined in \cite[Definition 4.9]{CELW}.
These then induce maps $w_r\colon BW_r^{p,p+n-1}(A)\rightarrow ZW_r^{p,p+n}(A)$.
  The \textit{$r$-page of the associated spectral sequence} also denoted $E_r^{p,p+n}(A)$ is then given by
  \begin{equation*}
    E_r^{p,p+n}(A)\coloneqq \frac{ZW_r^{p,p+n}(A)}{w_r\left(BW_r^{p,p+n-1}(A)\right)}
    =\frac{\Hom_{\bCh}(\ZW_r(p,p+n),A)}{\phi^\ast\Hom_{\bCh}(\BW_r(p,p+n-1),A)}\;.
  \end{equation*}
  The differential $d$ restricts to $r$-cycles and there is then an induced differential on the $r$-page sending a class $[(a_0,a_1,\ldots,a_{r-1})]$ to $[(d_1a_{r-1},0,0,\ldots,0)]$.
\end{defi}

\begin{defi}
  Given a morphism $f\colon A\rightarrow B$ of $\bCh$ we will say $f$ is \textit{$ZW_k$-bidegreewise surjective (resp.\ injective)} if $ZW_k^{p,p+n}(f)\colon ZW_k^{p,p+n}(A)\rightarrow ZW_k^{p,p+n}(B)$ is surjective (resp.\ injective) for all $p,n\in\Zbb$ and similarly with the functors $BW_k^{\ast,\ast}$ and $E_k^{\ast,\ast}$ as well.
\end{defi}

\begin{defi}
  The class of morphisms of bicomplexes inducing an isomorphism between the $(r+1)$-pages of the associated spectral sequences is denoted $\Er$.
  These will be called the \textit{$r$-quasi-isomorphisms} or \textit{$r$-weak equivalences}.
\end{defi}

\begin{defi}
  A filtered chain complex $A$ is said to be \textit{$r$-acyclic} if the $(r+1)$-page of its associated spectral sequence is $0$.
\end{defi}
  
There are $r$-homotopical analogues of the cone construction in chain complexes.
The following appears as \cite[Definition 3.5]{CELW} although with different sign conventions.
\begin{defi}\label{rConeFilteredChains}
  For a morphism $f\colon A\rightarrow B$ of filtered chain complexes the \textit{$r$-cone} of $f$ denoted $C_r(f)$ is the filtered chain complex with underlying filtered graded module given by $\rSusp A\oplus B$, in particular $F_pC_r(f)\coloneqq F_{p-r}A^{n+1}\oplus F_pB^n$ and differential given by $d\colon (a,b) \mapsto (-da,fa+db)$.
  We denote by $C_r(A)$ the $r$-cone of the identity morphism $\id\colon A\rightarrow A$. 
\end{defi}

\begin{lemm}[{\cite[Remark 3.6]{CELW}}]\label{rConeIsrAcyclic}
  Let $f\colon A\rightarrow B$ be a morphism of filtered chain complexes.
  Then $f\in\Er$ if and only if $C_r(f)$ is $r$-acyclic.\qed
\end{lemm}
\begin{lemm}[{\cite[Notation 3.7]{CELW}}]\label{loopFibrationsSurjective}
  There is a morphism of filtered chain complexes $\pi_1\colon \rLoops C_r(A)\rightarrow A$ for which $Z_k(\pi_1)$ is bidegreewise surjective for $0\leq k\leq r$.\qed
\end{lemm}

In \cite[\S 4.3]{CELW} the authors also construct $r$-cone objects of morphisms of bicomplexes by defining $r$-homotopical cylinder objects.
We only need the $r$-cone of the identity so we bypass these definitions and take as our definition the following.
First write $C_0(R^{0,0})$ for the bicomplex:
\begin{equation*}
  \begin{tikzcd}
    R^{0,0}\\
    R^{0,-1}\arrow[u,"1"]
  \end{tikzcd}\,.
\end{equation*}
\begin{defi}
  A bicomplex $A$ is said to be \textit{$r$-acyclic} if the $(r+1)$-page of its associated spectral sequence is $0$.
\end{defi}

\begin{defi}
  The \textit{$r$-cone} of a bicomplex $A$ is given:
  \begin{itemize}
  \item for $r=0$ by $C_0(R^{0,0})\otimes A$, and
  \item for $r\geq 1$ by $\ZW_r(r,r-1)\otimes A$.
  \end{itemize}
\end{defi}
\begin{lemm}[{\cite[Proposition 4.30]{CELW}}]\label{rConeIsAcyclicBicomplexes}
  Let $A$ be a bicomplex. Then $C_r(A)$ is $r$-acyclic.\qed
\end{lemm}
\begin{lemm}[{\cite[Proposition 4.32]{CELW}}]\label{loopFibrationSurjectiveBicomplexes}
  \sloppy There is a morphism of bicomplexes $\psi_r\colon C_r(A)\rightarrow \rSusp A$ for which $\ZW_k(\psi_r)$ is bidegreewise surjective for all $0\leq k\leq r$.\qed
\end{lemm}
\subsection{Model categories}
We take as our definition of a \textit{model category} that of \cite[Definition 1.1.3]{H}, in particular it is a complete and cocomplete category with functorial factorisations, and take \cite[Definition 2.1.17]{H} as our definition of a \textit{cofibrantly generated model category}.

    In a cofibrantly generated model category the \textit{generating cofibrations} will be denoted $I$ and \textit{generating acyclic cofibrations} $J$ (we will frequently adorn these with subscripts to specify a particular model category).
    For a class $K$ of morphisms we write $K\mathdash\Inj$ (resp.\ $K\mathdash\Proj$) for the the $K$-injectives (resp.\ $K$-projectives), i.e.\ those morphisms having the right (resp.\ left) lifting property with respect to $K$ and set $K\mathdash\Cof\coloneqq (K\mathdash\Inj)\mathdash\Proj$.
    We write $K\mathdash\Cell$ for the \textit{$K$-cellular morphisms}, i.e.\ those built as transfinite compositions of pushouts of elements of $K$, see \cite[Definition 2.1.9]{H}.

    We additionally say a model category is \textit{finitely cofibrantly generated} if the domains and codomains of $I$ and $J$ are finite relative to the cofibrations, \cite[\S 7.4]{H}.

We list some standard definitions of desirable properties of model categories.
\begin{defi}\label{modelCategoryProperties}
  A model category $\mc{M}$ is said to be
  \begin{itemize}
  \item \textit{left proper} if pushouts of weak equivalences along cofibrations are weak equivalences,
  \item \textit{right proper} if pullbacks of weak equivalences along fibrations are weak equivalences,
  \item \textit{proper} if $\mc{M}$ is both left proper and right proper,
  \item \textit{cellular} if it is a cofibrantly generated model category and
    \begin{enumerate}
    \item the domains and codomains of $I$ are small,
    \item the domains of $J$ are small relative to $I$,
    \item \sloppy the cofibrations are \textit{effective monomorphisms}, i.e.\ any cofibration $i\colon X\rightarrow Y$ is the equaliser of the two inclusions
      \begin{equation*}
        \begin{tikzcd}
          Y\arrow[r,shift left]\arrow[r,shift right] & Y\coprod_X Y\,,
        \end{tikzcd}
      \end{equation*}
    \end{enumerate}
  \item \textit{pointed} if the initial and terminal objects are isomorphic (in which case we denote either by $\ast$).
  \end{itemize}
\end{defi}
Given a model category $\mc{M}$ one can form the \textit{Reedy model structure} on the categories of simplicial objects in $\mc{M}$ and cosimplicial objects of $\mc{M}$, see \cite[Theorem 5.2.5]{H}.
These lead to the definitions of (left and right) \textit{homotopy function complexes} $\map_l(-,-)$ and $\map_r(-,-)$ which are simplicial sets, see \cite[Corollary 5.4.4]{H}.
Their right derived functors are isomorphic by \cite[Theorem 5.4.9]{H} so for $A$ cofibrant and $X$ fibrant we have $\map_l(A,X)\simeq\map_r(A,X)$ and in this case we write $\map(-,-)$ for either choice.
The use of homotopy function complexes in this paper will be in dealing with Bousfield localisations.

For computing the \textit{homotopy loops functor} we make use of the following lemma whose proof can be found in \cite[Corollary 13.3.8]{Hirschhorn}.
\begin{lemm}\label{loopsViaPullback}
  Let $\mc{M}$ be a right proper pointed model category, $f$ be a fibration, and $Y\simeq \ast$.
  Then the following pullback diagram gives a model for the homotopy pullback of $X$ by two points
  \begin{equation*}
    \begin{tikzcd}
      Z\arrow[r,dashed]\arrow[d,dashed]\arrow[dr,very near start, "\lrcorner", phantom]&Y\arrow[d,"f", two heads]\\
      \ast\arrow[r]&X
    \end{tikzcd}
  \end{equation*}
  i.e.\ $Z$ is a model for the loops functor applied to $X$, $Z\simeq\Omega X$.\qed
\end{lemm}
There is an analogous statement for constructing homotopy pushouts in a left proper model category.
\begin{defi}\label{stable}
  A pointed model category $\mc{M}$ is said to be stable if the suspension and loops functors give an equivalence of categories on the homotopy category.
\end{defi}
\subsection{Bousfield localisations}
Left (resp.\ right) Bousfield localisations provide a method of constructing a new model category from an existing one by expanding the class of weak equivalences and retaining the same cofibrations (resp.\ fibrations).
Existence of a left or right Bousfield localisation is not guaranteed, however under some reasonable conditions they can be constructed.
The definitions and results in this section can mostly be found in \cite{Hirschhorn}.
\begin{defi}
  Let $\mc{M}$ be a model category and $\mc{C}$ be a subclass of its morphisms, then:
  \begin{itemize}
  \item an object $W$ of $\mc{M}$ is \textit{$\mc{C}$-local} if it is fibrant and for all $f\colon A\rightarrow B$ of $\mc{C}$ the induced maps on the homotopy function complexes, $\map(B,W)\rightarrow \map(A,W)$, is a weak equivalence of simplicial sets, and
  \item a morphism $g\colon X\rightarrow Y$ of $\mc{M}$ is a \textit{$\mc{C}$-local equivalence} if for every $\mc{C}$-local object $W$ of $\mc{M}$ the induced maps of homotopy function complexes, $\map(g,W)\colon \map(Y,W)\rightarrow \map(X,W)$, are weak equivalences of simplicial sets.
  \end{itemize}
\end{defi}
\begin{defi}
  The \textit{left Bousfield localisation} of a model category $\mc{M}$ at a class of morphisms $\mc{C}$ denoted $L_{\mc{C}}\mc{M}$ is, if it exists, a model category with underlying category that of $\mc{M}$ and
  \begin{itemize}
  \item weak equivalences the $\mc{C}$-local equivalences of $\mc{M}$, and
  \item cofibrations the cofibrations of $\mc{M}$.
  \end{itemize}
\end{defi}
\begin{theo}[{\cite[Theorem 4.1.1]{Hirschhorn}}]
  Let $\mc{M}$ be a left proper and cellular model category and $\mc{C}$ a subset of morphisms, then the left Bousfield localisation of $\mc{M}$ at $\mc{C}$ exists.\qed
\end{theo}
\begin{defi}
  Let $\mc{M}$ be a model category and $\mc{K}$ a subclass of objects of $\mc{M}$.
  A morphism $g\colon X\rightarrow Y$ is said to be a \textit{$\mc{K}$-colocal equivalence} or a \textit{$K$-cellular equivalence} if for all $K\in\mc{K}$ the induced morphisms of homotopy function complexes $\map(K,g)\colon\map(K,X)\rightarrow\map(K,Y)$ are weak equivalences of simplicial sets.
\end{defi}
\begin{defi}
  The \textit{right Bousfield localisation} of a model category $\mc{M}$ at a class of objects $\mc{K}$ denoted either $R_{\mc{K}}\mc{M}$ or $\mc{K}\mathdash\cellularization\mathdash\mc{M}$ is, if it exists, a model category with underlying category that of $\mc{M}$ and
  \begin{itemize}
  \item weak equivalence the $\mc{K}$-colocal equivalences of $\mc{M}$, and
  \item fibrations the fibrations of $\mc{M}$.
  \end{itemize}
\end{defi}
\begin{theo}[{\cite[Theorem 5.1.1]{Hirschhorn}}]
  Let $\mc{M}$ be a right proper and cellular model category and $\mc{K}$ a subset of objects, then the right Bousfield localisation of $\mc{M}$ at $\mc{K}$ exists.\qed
\end{theo}
We will make use of the following detection of weak equivalences result in left proper, cofibrantly generated model categories due to Dugger.
\begin{prop}[{\cite[Proposition A.5]{Dugger}}]\label{duggersTheorem}
  Let $\mc{M}$ be a left proper, cofibrantly generated model category.
  Then there exists a set $W$ of cofibrant objects of $\mc{M}$ detecting weak equivalences, i.e.\ $X\rightarrow Y$ is a weak equivalence if and only if the induced map on homotopy function complexes $\map(A,X)\rightarrow\map(A,Y)$ is a weak equivalence of simplicial sets for all $A\in W$.
  Furthermore the set $W$ can be taken to be cofibrant replacements of the domains and codomains of the generating cofibrations.\qed
\end{prop}
\subsection{Greenlees and Shipley's Cellularization Principle}
The Cellularization Principle of Greenlees and Shipley is a result giving a Quillen equivalence between two localised model categories given an adjunction between the (unlocalised) model categories and a set of localising objects (along with some assumptions on the model categories and set of objects).
This result can be found in \cite{GS}.
\begin{defi}
  Let $\mc{M}$ be a model category. We say an object $K$ of $\mc{M}$ is \textit{homotopically small} if for any set of objects $\{Y_\alpha\}$ we have, in the homotopy category of $\mc{M}$, the natural map $\bigoplus_\alpha [K,Y_\alpha]\rightarrow [K,\bigwedge_\alpha Y_\alpha]$ is an isomorphism.
\end{defi}
\begin{theo}[The Cellularization Principle]\label{cellularizationPrinciple}
  Let $\mc{M}$ and $\mc{N}$ be right proper, stable, cellular model categories with a Quillen adjunction $\inadj{F}{\mc{M}}{\mc{N}}{U}$.
  
  Write $Q$ and $R$ for cofibrant and fibrant replacement functors.
  \begin{enumerate}
  \item Let $\mc{K}=\{K_\alpha\}$ be a set of objects in $\mc{M}$ and $FQ\mc{K}\coloneqq\{FQK_\alpha\}$ the corresponding set in $\mc{N}$.
    Then $F$ and $U$ induce a Quillen adjunction
    \begin{equation*}
      \adj{F}{\mc{K}\mathdash\cellularization\mathdash\mc{M}}{FQ\mc{K}
        \mathdash\cellularization\mathdash\mc{N}}{U}
    \end{equation*}
    between the $\mc{K}$-cellularization of $\mc{M}$ and the $FQ\mc{K}$-cellularization of $\mc{N}$.
  \item Furthermore if $\mc{K}$ is a stable set of homotopically small objects such that for each $K_\alpha\in\mc{K}$ the object $FQK_\alpha$ is homotopically small in $\mc{N}$ and the derived unit $QK_\alpha\rightarrow URFQK_\alpha$ is a weak equivalence for each $K_\alpha\in\mc{K}$ then $F\dashv U$ induces a Quillen equivalence
    \begin{equation*}
     \mc{K}\mathdash\cellularization\mathdash\mc{M} \simeq_QFQ\mc{K}\mathdash\cellularization\mathdash\mc{N}
    \end{equation*}
  \end{enumerate}
\end{theo}
\begin{proof}
  \cite[Theorem 2.7]{GS}.
\end{proof}
\subsection{$S$-model structures on $\fCh$ and $\bCh$}
In \cite{CELW} the authors introduced model structures, for each $r\geq 0$, on filtered chain complexes and bicomplexes whose weak equivalences were the $r$-quasi-isomorphisms of the associated spectral sequence.
These model structures came in two flavours: for filtered chain complexes the fibrations of the first were those morphisms that were bidegreewise surjective on the $r$-cycles, fibrations in the second were bidegreewise surjective on all $s$-cycles for $s\leq r$.
The case for bicomplexes replaces cycles with witness cycles and additionally imposed $0$-witness cycle surjectivity in the first model structure.

In \cite{BThesis} the author of this paper then found more model structures between the two types just outlined.
We recall those here.

We introduce the convention that in the context of filtered chain complexes $S$ is a finite subset of $\mathbb{N}\cup\{0\}$ with $r=\max S$.
In the context of bicomplexes $S$ additionally has $0$ as an element.

\begin{defi}
  For the category of filtered chain complexes and $S\subseteq \left\{0,1,\ldots,r\right\}$ containing $r$ we let:
  \begin{align*}
    I_r&\coloneqq \left\{\phi_{r+1}\colon \Z_{r+1}(p,n)\rightarrow \B_{r+1}(p,n)\right\}_{p,n\in\Zbb}\\
    J_r&\coloneqq\left\{0\colon 0\rightarrow \Z_r(p,n)\right\}_{p,n\in\Zbb}
  \end{align*}
  and more generally define:
  \begin{align*}
    I_S&\coloneqq I_r\cup\bigcup_{s\in S}J_s\,,&
                                              J_S&\coloneqq \bigcup_{s\in S}J_s\,.
  \end{align*}
\end{defi}

\begin{theo}\label{filteredChainsSModelCategories}
  For every $r\geq 0$ and subset $S\subseteq\{0,1,2,\ldots,r\}$ including $r$, the category $\fCh$ admits a right proper, finitely cofibrantly generated model structure, which we denote $\fChS$, where
  \begin{enumerate}
  \item weak equivalences are the $r$-quasi-isomorphisms,
  \item fibrations are the morphisms that are bidegreewise surjective on $s$-cycles for all $s\in S$, and
  \item $I_S$ and $J_S$ are the generating cofibrations and generating acyclic cofibrations respectively.
  \end{enumerate}
\end{theo}
\begin{proof}
  For $S=\left\{r\right\}$ or $\left\{0,1,\ldots,r\right\}$ the model structures were shown in \cite[Theorems 3.14 and 3.16]{CELW} and for other $S$ in \cite[Theorem 3.0.3]{BMonoidal}.
  The finitely cofibrantly generated property follows from \cref{FiniteFilteredChains}.
\end{proof}
\begin{defi}
  For the category of bicomplexes and $S\subseteq\left\{0,1,\ldots,r\right\}$ containing both $0$ and $r$ we let:
  \begin{align*}
    I_r&\coloneqq \left\{\phi_{r+1}\colon \ZW_{r+1}(p,p+n)\rightarrow \BW_{r+1}(p,p+n-1)\right\}_{p,n\in\Zbb}\\
    J_r&\coloneqq\left\{0\colon 0\rightarrow \ZW_r(p,p+n)\right\}_{p,n\in\Zbb}\cup
         \left\{0\colon 0\rightarrow \ZW_0(p,p+n)\right\}_{p,n\in\Zbb}
  \end{align*}
  and more generally define:
  \begin{align*}
    I_S&\coloneqq I_r\cup\bigcup_{s\in S}J_s\,,&
                                                 J_S&\coloneqq \bigcup_{s\in S}J_s\,.
  \end{align*}
\end{defi}
\begin{theo}\label{bicomplexesSModelCategories}
  For every $r\geq 0$ and subset $S\subseteq\{0,1,2,\ldots,r\}$ including both $0$ and $r$, the category $\bCh$ admits a right proper, finitely cofibrantly generated model structure, which we denote $\bChS$, where
  \begin{enumerate}
  \item weak equivalences are the $r$-quasi-isomorphisms,
  \item fibrations are the morphisms that are bidegreewise surjective on $s$-witness cycles for all $s\in S$, and
  \item $I_S$ and $J_S$ are the generating cofibrations and generating acyclic cofibrations respectively.
  \end{enumerate} 
\end{theo}
\begin{proof}
  For $S=\left\{0,r\right\}$ or $\left\{0,1,\ldots,r\right\}$ the model structures were shown in \cite[Theorems 4.37 and 4.39]{CELW} and for other $S$ the proof follows in the same way as \cite[Theorem 3.0.3]{BMonoidal}.
  The finitely cofibrantly generated property follows from \cref{smallBicomplexes}.
\end{proof}

For a fixed $r$ the model categories $\fChS$ as $S$ varies are distinct model categories, in that they have different fibrations and hence different cofibrations.
The following morphisms can be assembled to give morphisms that are $Z_k$-bidegreewise surjective for all $k\neq s$.
\begin{defi}
  The morphisms $\alpha_s^{p,p+n}$ and $\beta_s^{p,p+n}$ of filtered chain complexes are given by:
  \begin{itemize}
  \item $\alpha_s^{p,p+n}\colon\Z_{s+1}(p+1,p+1+n)\rightarrow \Z_s(p,p+n)$ whose underlying maps of $R$-modules are the identity wherever possible, and
  \item $\beta_s^{p,p+n}\colon\Z_{s-1}(p,p+n)\oplus R_{(p-s)}^{n+1}\rightarrow\Z_s(p,p+n)$ whose underlying maps of $R$-modules are the fold map or identity wherever possible.
  \end{itemize}
\end{defi}
\begin{lemm}We have that:
  \begin{itemize}
  \item the morphisms $\alpha_s^{\ast,\ast}$ are $Z_k$-bidegreewise surjective for all $k\geq s+1$ and not $Z_k$-bidegree\-wise surjective otherwise, and
  \item the morphisms $\beta_s^{\ast,\ast}$ are $Z_k$-bidegreewise surjective for all $k\leq s-1$ and not $Z_k$-bidegree\-wise surjective otherwise.\qed
  \end{itemize}
\end{lemm}
One can show the following morphism:
\begin{equation*}
  \gamma_s^{p,p+n}\coloneqq \nabla\circ\left(\alpha_s^{p,p+n}\oplus\beta_s^{p,p+n}\right)
\end{equation*}
is $Z_k$-bidegreewise surjective for all $k\neq s$.
This shows the model structures $\fChS$ all have different classes of fibrations.
One can construct similar morphisms for bicomplexes demonstrating an analogous result.
This is done in \cite[\S 3.3]{BThesis}.

For a set $S$ write $S\pm l\coloneqq \left\{s\pm l\,|\, s\in S\right\}$.
\begin{prop}\label{SDecQEquiv}
  There are Quillen equivalences $\inadj{S^l}{\fChS}{\left(\fCh_{S+l}\right)}{\Dec^l}$.
\end{prop}
\begin{proof}
  The $S=\left\{r\right\}$ case was shown in \cite[Theorem 3.22]{CELW} and the generalisation to any $S$ has the same proof.
\end{proof}
\begin{prop}\label{ididQEquiv}
  let $T\subseteq S$ with $\max T=\max S = r$.
  There are Quillen equivalences $\inadj{\id}{\left(\fCh\right)_T}{\fChS}{\id}$,
  and if additionally $0\in T\subseteq S$ there are Quillen equivalences $\inadj{\id}{\left(\bCh\right)_T}{\bChS}{\id}$.
\end{prop}
\begin{proof}
  For filtered chains, $T=\left\{0\right\}$ and $S=\left\{0,1,\ldots,r\right\}$ this was observed in \cite[Remark 3.17]{CELW} and more generally this follows by the same observation.
  The case for bicomplexes is the same.
\end{proof}


%% file: tex/adjoints.tex
\section{Adjoints to totalisation functors}\label{adjoints}
We construct a left adjoint $\mc{L}$ to the product totalisation functor $\Totp$, show that $\mc{L}$ applied to a representing $s$-cycle decomposes as a direct sum of an $s$-witness cycle and infinitely many copies of $0$-witness cycles, and finally show that the unit of the adjunction applied to an $s$-cycle is an $s$-quasi-isomorphism when $s\geq 1$.

We will use these facts later in applying Greenlees and Shipley's cellularization principle to show there are Quillen equivalences between $\fChS$ and $\bChS$.

\subsection{Left adjoint to the product totalisation functor}
\begin{defi}
  The functor $\mc{L}\colon\fCh\rightarrow\bCh$ is defined on a filtered chain complex $A$ by:
  \begin{equation*}
    \mc{L}^{i,i+n}(A)\coloneqq \frac{A^n}{F_{i-1}A^n}\oplus\frac{A^{n-1}}{F_iA^{n-1}}
  \end{equation*}
  where the differentials $d_0$ and $d_1$ are given on an $(x,y)\in\mc{L}(A)^{i,i+n}$ by:
  \begin{align*}
    d_0\colon(x,y)&\mapsto(dx,x-dy)\;,\\
    d_1\colon(x,y)&\mapsto(0,(-1)^{n+1}x)\;.
  \end{align*}
  On a morphism $f\colon A\rightarrow B$ of filtered chain complexes the functor $\mc{L}$ is given by $\mc{L}(f)^{i,i+n}\coloneqq \bar{f}_i^n\oplus\bar{f}_{i+1}^{n-1}$ where $\bar{f}_i^n$ denotes the induced morphism of $R$-modules $A^n/F_{i-1}A^n\rightarrow B^n/F_{i-1}B^n$.
\end{defi}
\begin{prop}\label{LTotpAdjunction}
  There is an adjunction $\inadj{\mc{L}}{\fCh}{\bCh}{\Totp}$.
\end{prop}
\begin{proof}
  Given a morphism $f\colon \mc{L}A\rightarrow B$ in $\bCh$ one can produce a morphism of filtered chain complexes $\tilde{f}\colon A\rightarrow \Totp(B)$ by:
  \begin{equation*}
    \begin{tikzcd}
      A^n\arrow[rr,dashed, "\tilde{f}"]\arrow[dr,"q\circ\Delta"']& &\Totp(B)^n=\prod_i B^{i,i+n}\\
      &\prod_i\frac{A^n}{F_{i-1}A^n}\arrow[ur, "{(f^{i,i+n}(-,0))_i}"']
    \end{tikzcd}
  \end{equation*}
  where the map $q\circ \Delta$ is the composite $A^n\rightarrow \prod_i A^n\rightarrow \prod_i A^n/F_{i-1}A^n$ of an infinite diagonal followed by quotients in each component.

  Conversely given a morphism $g\colon A\rightarrow \Totp(B)$ of filtered chain complexes one obtains a morphism of bicomplexes $\hat{g}\colon \mc{L}(A)\rightarrow B$ defined by:
  \begin{align*}
    \hat{g}^{i,i+n}\colon \mc{L}(A)^{i,i+n}&\rightarrow B^{i,i+n}\\
    (\bar{x},\bar{y})&\mapsto g^{i,i+n}(x)+(-1)^nd_1g^{i+1,i+1+n-1}(y)
  \end{align*}
  where $x$ and $y$ are any lifts of the element $(\bar{x},\bar{y})\in \mc{L}(A)^{i,i+n}=A^n/F_{i-1}A^n\oplus A^{n-1}/F_iA^{n-1}$ to an element of $(x,y)\in A^n\oplus A^{n-1}$.
  This is a well defined morphism of bicomplexes.

  The constructions $f\mapsto \tilde{f}$ and $g\mapsto \hat{g}$ can be verified to be inverse to each other.
\end{proof}
The adjunction of \cref{LTotpAdjunction} is clearly not an equivalence of categories: consider a non-zero object $A\in\fCh$ whose filtration is such that $F_pA=A$ for all $p$.
The bicomplex $\mc{L}(A)$ is then the $0$ bicomplex so any morphism out of it also $0$, however there are non-zero morphisms with domain $A$ in the category $\fCh$.
\subsection{Right adjoint to the coproduct totalisation functor}
The coproduct totalisation functor $\Totc$, where one replaces $\prod$ with $\bigoplus$, also has a right adjoint $\mc{R}$.
The proof of this adjunction is dual to that for $\mc{L}\dashv\Totp$.
We state the definition for completeness but make no use of it in this paper.
\begin{defi}
  The functor $\mc{R}\colon\fCh\rightarrow\bCh$ is defined on a filtered chain complex $A$ by:
  \begin{equation*}
    \mc{R}^{i,i+n}(A)\coloneqq F_{i-1}C^{n+1}\oplus F_iC^n
  \end{equation*}
  where the differentials $d_0$ and $d_1$ are given on an $(x,y)\in\mc{R}(A)^{i,i+n}$ by:
  \begin{align*}
    d_0\colon(x,y)&\mapsto(-dx,x+dy)\;,\\
    d_1\colon(x,y)&\mapsto(0,(-1)^{n+1}x)\;.
  \end{align*}
  On a morphism $f\colon A\rightarrow B$ of filtered chain complexes the functor $\mc{R}$ is given by $\mc{R}(f)^{i,i+n}\coloneqq f_{|i-1}^{n+1}\oplus f_{|i}^n$ where $f_{|i}^n$ denotes the morphism of $R$-modules $F_iA^n\rightarrow F_iB^n$.
\end{defi}
\begin{prop}
  There is an adjunction $\inadj{\Totc}{\bCh}{\fCh}{\mc{R}}$.\qed
\end{prop}
\subsection{The left adjoint applied to $s$-cycles}
\begin{lemm}\label{LOfSCycle}
  There is an isomorphism of bicomplexes
  \begin{equation*}
    \mc{L}\Z_s(p,n)\cong
    \begin{cases}
      \bigoplus_{k\geq0}\ZW_0(p-s-k,p-s-k+n)\, & s=0,\\
      \ZW_s(p,p+n)\oplus\bigoplus_{k\geq0}\ZW_0(p-s-k,p-s-k+n)\,, & s\geq 1.
    \end{cases}
  \end{equation*}
\end{lemm}
\begin{proof}
  The bicomplex $\mc{L}\Z_s(p,n)$ is depicted in \cref{figu:LZs} which appears similar to a copy of $\ZW_s(p+n,n)$ glued with an infinite number of copies of shifted $\ZW_0(\ast,\ast)$.
  By taking a change of basis bidegreewise one obtains the isomorphism of the lemma.
  \begin{figure}[h]
    \centering
    \begin{tikzpicture}[xscale=1.7,yscale=1.4]
      \node (00) at (0,0) {$R$};
      \node (-1-1) at (-1,-1) {$R$};
      \node (-2-2) at (-2,-2) {$R$};
      \node (-3-3) at (-3,-3) {$R$};
      \node (-4-4) at (-4,-4) {$R$};
      \node (-5-5) at (-5,-5) {$R$};
      \node (-6-6) at (-6,-6) {$\iddots$};
      \node (-10) at (-1,0) {$R$};
      \node (-2-1) at (-2,-1) {$\iddots$};
      \node (-3-2) at (-3,-2) {$R$};
      \node (-4-3) at (-4,-3) {$R\oplus R$};
      \node (-5-4) at (-5,-4) {$R\oplus R$};
      \node (-6-5) at (-6,-5) {$R\oplus R$};
      \node (-7-6) at (-7,-6) {$\iddots$};
      \node (-5-3) at (-5,-3) {$R$};
      \node (-6-4) at (-6,-4) {$R$};
      \node (-7-5) at (-7,-5) {$R$};
      \draw[->] (00) -- node[anchor=south] {$(-1)^{n+1}$} (-10);
      \draw[->] (-1-1) -- node[anchor=south] {$(-1)^{n+1}$} (-2-1);
      \draw[->] (-2-2) -- node[anchor=south] {$(-1)^{n+1}$} (-3-2);
      \draw[->] (-3-3) -- node[anchor=north] {$(-1)^{n+1}i_1$} (-4-3);
      \draw[->] (-4-4) -- node[anchor=north] {$(-1)^{n+1}i_1$} (-5-4);
      \draw[->] (-5-5) -- node[anchor=north] {$(-1)^{n+1}i_1$} (-6-5);
      \draw[->] (-1-1) -- node[anchor=west] {$1$} (-10);
      \draw[->] (-2-2) -- node[anchor=west] {$1$} (-2-1);
      \draw[->] (-3-3) -- node[anchor=west] {$1$} (-3-2);
      \draw[->] (-4-4) -- node[anchor=east] {$i_1$} (-4-3);
      \draw[->] (-5-5) -- node[anchor=east] {$i_1$} (-5-4);
      \draw[->] (-4-3) -- node[anchor=south] {$(-1)^n\pi_1$} (-5-3);
      \draw[->] (-5-4) -- node[anchor=south] {$(-1)^n\pi_1$} (-6-4);
      \draw[->] (-6-5) -- node[anchor=south] {$(-1)^n\pi_1$} (-7-5);
      \draw[->] (-5-4) -- node[anchor=east] {$\begin{pmatrix}0&-1\end{pmatrix}$} (-5-3);
      \draw[->] (-6-5) -- node[anchor=east] {$\begin{pmatrix}0&-1\end{pmatrix}$} (-6-4);
    \end{tikzpicture}
    \caption{The bicomplex $\mc{L}\Z_s(p,n)$}
    \label{figu:LZs}
  \end{figure}
\end{proof}
\begin{prop}\label{unitIsSIsomorphism}
  For $s\geq 1$ the unit of the adjunction $\mc{L}\dashv\Totp$ applied to an $s$-cycle $\Z_s(p,n)\rightarrow \Totp\mc{L}\Z_s(p,n)$ is an isomorphism on the $s$-page.
\end{prop}
\begin{proof}
  This follows from \cref{LOfSCycle}: one can compute that the unit of the adjunction $\Z_s(p,n)\rightarrow \Totp\mc{L}\Z_s(p,n)$, being the adjunct of $\id\colon \mc{L}\Z_s(p,n)\rightarrow \mc{L}\Z_s(p,n)$, sends the element $1_{(p)}^n$, a generator of the copy of $R_{(p)}^n$ in $\Z_s(p,n)$, to the infinite diagonal of $1$s in $\Totp\mc{L}\Z_s(p,n)=\prod_{k\leq p}R$.
  The $s$-page of both $\Z_s(p,n)$ and $\Totp\mc{L}\Z_s(p,n)$ both have a copy of $R$ in bidegrees $(p,p+n)$ and $(p-s-1,p+n-s)$ with an identity differential between them (the case for $\Totp\mc{L}\Z_s(p,n)$ follows from \cref{LOfSCycle}).
  
  The infinite product of $1$s we've just described is a generator of the copy of $R$ in bidegree $(p,p+n)$ of the spectral sequence just described hence the unit induces an isomorphism between the $s$-pages of the associated spectral sequences.
\end{proof}

%% file: tex/properness.tex
\section{Properness}\label{properness}
The model categories $\fChS$ and $\bChS$ are right proper by construction since every object is fibrant (a proof of this is dual to \cite[Lemma 9.5]{GJ}).
To demonstrate $\fChS$ and $\bChS$ are additionally left proper we make use of a method from a paper of Lack, \cite{Lack}.
Lack shows the category of \textit{$2$-categories} with \textit{$2$-functors} as morphisms is left proper.

We later use left properness to identify sets of cofibrant objects detecting weak equivalences in $\fChS$ and $\bChS$ using a result of Dugger.
\begin{prop}\label{LacksProposition}
  Let $\mc{M}$ be a finitely cofibrantly generated model category with generating cofibrations $I$ such that whenever there is a double pushout diagram of the form
  \begin{equation}\label{LacksPropernessDiagram}
    \begin{tikzcd}
      S\arrow[r]\arrow[d,"i"']
      \arrow[dr,phantom,very near end, "\ulcorner"]
      &A\arrow[d,"f"]\arrow[r,two heads, "\pi", "\sim"']
      \arrow[dr,phantom,very near end, "\ulcorner"]
      &B\arrow[d]\\
      D\arrow[r]&C\arrow[r,"\pi'"]&P
    \end{tikzcd}
  \end{equation}
  with $i$ a generating cofibration and $\pi$ an acyclic fibration then $\pi'$ is an acyclic fibration.
  Then $\mc{M}$ is a left proper model category.
\end{prop}
\begin{proof}
  The proof is done in three steps in \cite[Theorem 6.3]{Lack}.
\end{proof}
\subsection{Properness of $\fChS$}
We wish to demonstrate that the model categories $\fChS$ of \cref{filteredChainsSModelCategories} are left proper and will do so by demonstrating the conditions of \cref{LacksProposition} hold for $\mc{M}=\fChS$.
We consider then a double pushout diagram of the form
\begin{equation}\label{fChSPropernessDiagram}
  \begin{tikzcd}
    \Z_{r+1}(p,n)\arrow[r]\arrow[d,"\phi"']
    \arrow[dr,phantom,very near end, "\ulcorner"]
    &A\arrow[d,"f"]\arrow[r,two heads, "\pi", "\sim"']
    \arrow[dr,phantom,very near end, "\ulcorner"]
    &B\arrow[d]\\
    \B_{r+1}(p,n)\arrow[r]&A'\arrow[r,"\pi'"]&B'
  \end{tikzcd}
\end{equation}
in the model category $\fChS$.
\begin{prop}\label{PushoutSurjectivity}
  Let $\pi\colon A\rightarrow B$ be a morphism of filtered chain complexes and $\pi'$ as in Diagram \ref{fChSPropernessDiagram}. We have the following surjectivity results on cycles:
  \begin{enumerate}
  \item suppose $\pi$ is $Z_s$-bidegreewise surjective for some $s\leq r$, then the pushout $\pi'$ is $Z_s$-bidegreewise surjective,
  \item suppose $\pi$ is $Z_r$-bidegreewise surjective and an $r$-weak equivalence, then the pushout $\pi'$ is $Z_{r+1}$-bidegreewise surjective.
  \end{enumerate}
\end{prop}
\begin{proof}
  Most cycle surjectivity claims follow from the case for $\pi$.
  The remaining cases for cycles introduced in the pushout are checked in \cite[Proposition 3.7.1.3]{BThesis}.
\end{proof}
\begin{lemm}\label{KernelOfPushout}
  The kernel of the pushout $\pi'$ is $K=\ker(\pi\colon A\rightarrow B)$.
\end{lemm}
\begin{proof}
  This follows from the description of such a pushout given in \cite[Lemma 4.1.7]{BMonoidal} and that limits are computed filtration and cohomological degreewise, \cite[Remark 2.6]{CELW}.
\end{proof}
Recall the morphism $w_r\colon B_r^{\ast,\ast}(A)\rightarrow Z_r^{\ast,\ast}(A)$ of \cref{associatedSS}.
\begin{prop}\label{PushoutEr+1Injective}
  For $\pi$ an $r$-weak equivalence which is $Z_r$-bidegreewise surjective, the morphism $E_{r+1}(\pi')\colon E_{r+1}(A')\rightarrow E_{r+1}(B')$ between the $(r+1)$-pages of the associated spectral sequences is injective.
\end{prop}
\begin{proof}
  Let $z$ be an $(r+1)$-cycle representing a class of $E_{r+1}^{\ast,\ast}(A')$ whose image under $E_{r+1}(\pi')$ is $0$ so that there's a boundary $(c_0,c_1)\in B_{r+1}^{\ast,\ast}(B')$ with $\pi'(z)=w_{r+1}((c_0,c_1))$, where $c_0$ and $c_1$ are the constituent $r$-cycles of the boundary element.
  Since $\pi$ is an $r$-acyclic fibration $\pi'$ is $Z_r$-bidegreewise surjective by \cref{PushoutSurjectivity}.
  Let $e_0$ and $e_1$ be $r$-cycle lifts of $c_0$ and $c_1$.

  Then $z-w_{r+1}((e_0,e_1))$ is an $(r+1)$-cycle of $A'$ in the kernel of $\pi'$.
  By \cref{KernelOfPushout} the kernel of $\pi'$ is the kernel of $\pi$ which is acyclic since $\pi$ is an $r$-acyclic fibration, hence $z-w_{r+1}((e_0,e_1))$ is also an $(r+1)$-boundary, say equal to $(k_0,k_1)$ in $B_{r+1}^{\ast,\ast}(A')$.

  Hence $z=w_{r+1}((e_0+k_0,e_1+k_1))$ is an $(r+1)$-boundary which shows injectivity of $E_{r+1}(\pi')$.
\end{proof}
\begin{coro}\label{pPrimefChSIsAcyclicFibration}
  Let $\pi$ be an $S$-acyclic fibration of $\fChS$.
  The pushout of $\pi$ along an $S$-generating cofibration is an $S$-acyclic fibration.\qed
\end{coro}
\begin{proof}
  If the generating cofibration defining $\pi'$ is of the form $0\rightarrow\Z_s(\ast,\ast)$ the result is clear.

  If the generating cofibration is of the form $\phi_{r+1}\colon\Z_{r+1}(\ast,\ast)\rightarrow\B_{r+1}(\ast,\ast)$ and $s\in S$ so that $\pi$ is $Z_s$-surjective then so too is $\pi'$ by part 1 of \cref{PushoutSurjectivity}.
  Since $\pi$ is $Z_r$-surjective and an $r$-weak equivalence the pushout $\pi'$ is $Z_{r+1}$-surjective by part 2 of \cref{PushoutSurjectivity} and so, by \cite[Lemma 2.8]{CELW}, $E_{r+1}(\pi')$ is bidegreewise surjective.
  Lastly since $\pi$ is an $r$-acyclic fibration $E_{r+1}(\pi')$ is bidegreewise injective by \cref{PushoutEr+1Injective}.
\end{proof}
\begin{prop}
  The model categories $\fChS$ of \cref{filteredChainsSModelCategories} are left proper.
\end{prop}
\begin{proof}
  By \cref{LacksProposition} it suffices to show that $\fChS$ is a finitely cofibrantly generated model category and that in the double pushout of Diagram \ref{LacksPropernessDiagram} with $\phi$ a generating cofibration and $\pi$ an acyclic fibration in $\fChS$ that $\pi'$ is also an acyclic fibration.
  The model structure is finitely cofibrantly generated by \cref{filteredChainsSModelCategories} and that $\pi'$ is an $S$-acyclic fibration was shown in \cref{pPrimefChSIsAcyclicFibration}.
\end{proof}
\subsection{Properness of $\bChS$}
Verifying the $\bChS$ are left proper is similar to the $\fChS$.
One first identifies the pushout $A'$ of $A$ along a generating $S$-cofibration, identifies the $s$-cycles of $A'$ and checks the pushout $\pi'$ satisfies similar surjectivity conditions as in \cref{PushoutSurjectivity}.
After also checking the bicomplex analogue of \cref{PushoutEr+1Injective} one can check that the pushout of an $S$-acyclic fibration along an $S$-generating cofibration is still an $S$-acyclic fibration and finally appeal to the same theorem of Lack.

We state the analogous lemmas without proof here (proofs can be found in \cite[\S 3.7.2]{BThesis}.
Again we write $A'$ and $B'$ as in the pushout:
\begin{equation}\label{bChSPropernessDiagram}
  \begin{tikzcd}
    \ZW_{r+1}(p,p+n)\arrow[r]\arrow[d,"\phi"']
    \arrow[dr,phantom,very near end, "\ulcorner"]
    &A\arrow[d,"f"]\arrow[r,two heads, "\pi", "\sim"']
    \arrow[dr,phantom,very near end, "\ulcorner"]
    &B\arrow[d]\\
    \BW_{r+1}(p,p+n-1)\arrow[r]&A'\arrow[r,"\pi'"]&B'
  \end{tikzcd}
\end{equation}
in the category of bicomplexes.

\begin{lemm}
  If $\pi$ is $ZW_s$-bidegreewise surjective for some $s\leq r+1$ then so too is $\pi'$.\qed
\end{lemm}

\begin{lemm}
  The kernel of the pushout $\pi'$ is $K=\ker\left(\pi\right)$.\qed
\end{lemm}

\begin{lemm}
  \sloppy For $\pi$ an $r$-acyclic fibration the induced morphism $E_{r+1}(\pi')\colon E_{r+1}(A')\rightarrow E_{r+1}(B')$ between the $(r+1)$-pages of the associated spectral sequences is injective.\qed
\end{lemm}

\begin{coro}
  Let $\pi$ be an $S$-acyclic fibration of $\bChS$.
  The pushout of $\pi$ along an $S$-generating cofibration is an $S$-acyclic fibration.\qed
\end{coro}

\begin{prop}
  The model categories $\bChS$ of \cref{bicomplexesSModelCategories} are left proper.\qed
\end{prop}


%% file: tex/cellularity.tex
\section{Cellularity}\label{cellularity}
We show the $\fChS$ and $\bChS$ are cellular, an assumption needed for the Cellularization Principle.
Recall the definition of a cellular model category from \cref{modelCategoryProperties}.
We recall the notion of a regular morphism which in the category of filtered chain complexes coincide with the effective morphisms.
It is easier to show a morphism is a regular morphism however.
\begin{defi}
  A monomorphism $i\colon A\rightarrow B$ is a \textit{regular morphism} if it is the equaliser of some pair of morphisms $f,g\colon B\rightarrow C$.
\end{defi}
\begin{prop}\label{effectiveEqualsRegular}
  In a category with equalisers and cokernel pairs the class of regular mono\-morphisms coincides with the class of effective monomorphisms.\qed
\end{prop}
\begin{proof}
  The proof is dual to that of \cite[Proposition 2.5.7]{Borceux}.
\end{proof}
\subsection{Cellularity of $\fChS$}
\begin{defi}
  Let $f\colon A\rightarrow B$ be a morphism of $\fCh$.
  We say $f$ is a \textit{strict morphism} if whenever $f(a)\in F_pB^n$ then $a\in F_pA^n$.
\end{defi}
We say a morphism in $\fCh$ is an inclusion if it is after forgetting filtration.
The following was shown in \cite[Lemmas 4.2.2 and 4.2.3]{BMonoidal}.
Colimits of strict morphisms being computed degreewise can be easily checked by studying the definition of colimits in $\fCh$ given in \cite[Remark 2.6]{CELW}.
  
\begin{lemm}\label{filteredChainsCofibrationsFacts}
  Cofibrations of $(\fCh)_S$ are in particular inclusions and strict morphisms.
  Cokernels of strict morphisms are computed filtration degreewise and cohomological degreewise.\qed
\end{lemm}

\begin{prop}\label{filteredChainsAreCellular}
  The model categories $\fChS$ of \cref{filteredChainsSModelCategories} are cellular.
\end{prop}
\begin{proof}
  We need to demonstrate the cellular conditions from \cref{modelCategoryProperties}.
  By \cref{smallFilteredChains} all filtered chain complexes are small hence the smallness conditions are satisfied.
  It remains to show all cofibrations are effective monomorphisms.
  By \cref{effectiveEqualsRegular} it suffices to show they are regular monomorphisms instead.
  Cofibrations of $\fChS$ are monomorphisms by \cref{filteredChainsCofibrationsFacts}.
  Since they are in addition strict morphisms their cokernels are computed filtration and cohomological degreewise, hence they are the kernel of their cokernel, and so any cofibration is the equaliser of the $0$ morphism and its cokernel, hence a regular monomorphism.
\end{proof}
\subsection{Cellularity of $\bChS$}
\begin{lemm}
  Cofibrations in $\bChS$ are (bidegreewise split) monomorphisms.
\end{lemm}
\begin{proof}
  The proof is similar to that for the projective model structure on chain complexes, see \cite[Proposition 2.3.9]{H}.
  One replaces the disc object in chain complexes with a witness $0$-cycle in bicomplexes.
\end{proof}
\begin{prop}\label{bicomplexesAreCellular}
  The model categories $\bChS$ of \cref{bicomplexesSModelCategories} are cellular.
\end{prop}
\begin{proof}
  We need to demonstrate the cellular conditions from \cref{modelCategoryProperties}.
  By \cref{smallBicomplexes} all bicomplexes are small hence the smallness conditions are satisfied.
  It remains to show all cofibrations are effective monomorphisms.
  By \cref{effectiveEqualsRegular} it suffices to show they are regular monomorphisms instead and since $\bCh$ is an Abelian category any monomorphism $i$ is the kernel of its cokernel, hence the equaliser of the $0$ morphism and its cokernel.
\end{proof}

%% file: tex/stability.tex
\section{Stability}\label{stability}
Recall from \cref{loopsViaPullback} that the loops functor on an object $X$ can be computed by a regular pullback
\begin{equation*}
  \begin{tikzcd}
    Z\arrow[r,dashed]\arrow[d,dashed]\arrow[dr,very near start, "\lrcorner", phantom]&Y\arrow[d,"f", two heads]\\
    \ast\arrow[r]&X
  \end{tikzcd}
\end{equation*}
where $f$ is a fibration with $Y\simeq \ast$.
In both the model categories $\fChS$ and $\bChS$ we have good models for such a $Y$ given an $X$.
Recall there is in either category a morphism $\rLoops C_r(A)\rightarrow A$ which is bidegreewise surjective on $k$-(witness) cycles for all $0\leq k\leq r$ by \cref{loopFibrationsSurjective,loopFibrationSurjectiveBicomplexes} and with $\rLoops C_r(A)$ $r$-acyclic by \cref{rConeIsrAcyclic,rConeIsAcyclicBicomplexes}.
To compute $\Omega A$ we need only then compute the pullback of this fibration by a point.

In both $\fChS$ and $\bChS$ we'll show, on the level of the model category, that $\Omega A$ can be modelled by $\rLoops A$ which is invertible, hence that the suspension-loop adjunction is an equivalence.
For filtered chain complexes $\rLoops A$ will be exactly the pullback whereas for bicomplexes we identify the pullback as being $r$-quasi-isomorphic to $\rLoops A$.

We make use of stability in the construction of Quillen equivalences between $\fChS$ and $\bChS$.
\subsection{Stability of $\fChS$}
\begin{prop}
  The model categories $\fChS$ of \cref{filteredChainsSModelCategories} are stable model categories with $\Omega A\simeq \rLoops A$ and $\Sigma A\simeq \rSusp A$.
\end{prop}
\begin{proof}
  We need to compute the pullback
  \begin{equation}\label{loopsPullbackDiagram}
    \begin{tikzcd}
      P\arrow[r,dashed]\arrow[d,dashed]\arrow[dr,very near start, "\lrcorner", phantom]&\rLoops C_r(A)\arrow[d,"\rLoops\pi_r", two heads]\\
      0\arrow[r]&A
    \end{tikzcd}
  \end{equation}
  where $C_r(\rLoops A)\coloneqq A\oplus\rLoops A$ with a twisted differential.
  The morphism $\rLoops\pi_r$ is simply projection onto the $A$ summand and we immediately have that the pullback is (isomorphic to) $\rLoops A$.
  The functor $\rLoops$ is invertible hence the suspension functor can be modelled by its inverse $\rSusp A$.
\end{proof}
\subsection{Stability of $\bChS$}
The proof for bicomplexes is similar for $r=0$, however for $r\geq 1$ we show that $\rLoops A$ is $r$-quasi isomorphic to a pullback of the form of \cref{loopsPullbackDiagram}.
There is a projection map $\pi\colon \ZW_r(r,r-1)\rightarrow R^{r,r-1}$ which when tensored on the right by $A$ gives the morphism $\psi_r\colon C_r(A)\rightarrow \rSusp A$.
So $\rLoops\pi\otimes \id_A$ gives a morphism $\psi_r\colon\rLoops C_r(A) \rightarrow A$ which is $ZW_k$-bidegreewise surjective for $0\leq k\leq r$ with $r$-acyclic domain.

Write $\NW_r(r,r-1)$ for the pullback of $\pi\colon\ZW_r(r,r-1)\rightarrow R^{r,r-1}$ by the $0$ morphism. This is essentially the bicomplex $\ZW_r(r,r-1)$ where the module in bidegree $(r,r-1)$ has been deleted.

There is then a pullback diagram of the form:
\begin{equation*}
  \begin{tikzcd}
    \rLoops\NW_r(r,r-1)\otimes A\arrow[r,dashed]\arrow[d,dashed]
    \arrow[dr,phantom,"\lrcorner", very near start]
    & \rLoops C_r(A)\arrow[d,"\rLoops \psi_r"]\\
    0\arrow[r]&A
  \end{tikzcd}\,.
\end{equation*}
We also write $i\colon R^{0,0}\rightarrow \ZW_r(r,r-1)$ for the inclusion of the bidegree $(0,0)$ component, so there's a morphism $\rLoops i\colon R^{-r,-r+1}\rightarrow \rLoops\NW_r(r,r-1)$ and we now claim that the morphism
\begin{equation*}
  \rLoops i\otimes\id_A\colon \rLoops A=R^{-r,-r+1}\otimes A
  =\rLoops R^{0,0}\otimes A\longrightarrow \rLoops\NW_r(r,r-1)\otimes A 
\end{equation*}
is a $1$-quasi isomorphism.
\begin{prop}
  The model categories $\bChS$ of \cref{bicomplexesSModelCategories} are stable model categories with $\Omega A\simeq \rLoops A$ and $\Sigma A\simeq \rSusp A$.
\end{prop}
\begin{proof}
  For $r=0$ the pullback in question is exactly $\Omega^0A$.
  For $r\geq 1$ it suffices to show the above morphism $\rLoops i\otimes\id_A$ is a $1$-quasi isomorphism.
  Given that for a bicomplex $B$ that $E_0^{p,p+n}(B)=B^{p,p+n}$ and that the $0$-page differential is given by the vertical differential $d_0$ of $B$, this is a simple spectral sequence computation.

  Hence $\Omega A\simeq \rLoops A$ and since $\rLoops$ is an invertible endofunctor of $\bCh$ we have $\Sigma A\simeq \rSusp A$.
\end{proof}

%% file: tex/bousfieldLocalisations.tex
\section{Bousfield localisations of $\fChS$ and $\bChS$}\label{bousfieldLocalisations}
Given a morphism of spectral sequences $f\colon\left\{E_r^{\bullet,\bullet},d_r\right\} \rightarrow\left\{E_r^{\bullet,\bullet},d'_r\right\}$ if $f$ is an isomorphism of differential graded modules between the $r$-pages then it is an isomorphism between all subsequent pages.
Thus writing $\mc{E}_r$ for the class of $r$-quasi-isomorphisms, for either $\fCh$ or $\bCh$, there is a sequence of strict inclusions:
\begin{equation*}
  \mc{E}_0\subset\mc{E}_1\subset\mc{E}_2\subset\ldots\subset\mc{E}_r\subset\ldots\,.
\end{equation*}
Given then finite non-empty sets $S$ and $T$ indexing model structures on $\fCh$ or $\bCh$ with $\max S<\max T$ the $S$-weak equivalences are a subclass of the $T$-weak equivalences.
We can then ask ``is the $T$-model structure a Bousfield localisation of the $S$-model structure?''.

\subsection{Non-existence of certain left Bousfield localisations}
Given one of the model structures on $\fCh$ or $\bCh$ with $r$-weak equivalences we show there is no left Bousfield localisation with weak equivalences the $(r+1)$-weak equivalences.
\begin{prop}
  Let $\mc{M}_S$ be any of $\fChS$ or $\bChS$ where $r=\max S$.
  There is no left Bousfield localisation $\mc{M}_{\new}$ of $\mc{M}_S$ whose weak equivalences are the $(r+1)$-weak equivalences.
\end{prop}
\begin{proof}
  Suppose such a $\mc{M}_{\new}$ existed, so that the cofibrations, fibrations and weak equivalences of $\mc{M}_{\new}$ satisfy $\Cof_{\new}=\Cof_S$, $\Fib_\new\subset\Fib_S$, and $\mc{W}_{\new}=\mc{W}_{r+1}\supset\mc{W}_r=\mc{W}_S$.
  Note that $\mc{M}_{\new}$ at least has a set of generating cofibration $I_{\new}=I_S$ and that $I_{\new}\subset \mc{W}_{\new}$.
  
  Any $I_\new\mathdash\Cell$ morphism is an $(r+1)$-weak equivalence since transfinite compositions of weak equivalences are weak equivalences, \cite[Corollary 7.4.2]{H}.
  Since any cofibration is a retract of a cellular one all elements of $\Cof_\new$ are $(r+1)$-weak equivalences.
  This implies any morphism of $\mc{M}_\new$ is an $(r+1)$-weak equivalence which is false, hence no such left Bousfield localisation of $\mc{M}_S$ exists.
\end{proof}


%% file: tex/quillenEquivalence.tex
\section{Quillen equivalences between the $S$-model structures}
Using the results of the previous sections we establish Quillen equivalences between $\fChS$ and $\bChS$ in \cref{quillenEquivalencesFCHSandBCHS}.
Finally in \cref{distributiveLatticeN} we construct a distributive lattice structure on the set of model structures $\fChS$ indexed by the sets $S$ where $T\leq S$ if there are a composite of the left adjoints: identity functors or shift functors from $\fChT$ to $\fChS$.

\subsection{Quillen equivalences between $\fChS$ and $\bChS$}\label{quillenEquivalencesFCHSandBCHS}
In this section $S$ will denote a subset of $\{0,1,2,\ldots,r\}$ containing both $0$ and $r$ regardless of whether it indexes a model structure of $\fCh$ or $\bCh$ unless otherwise specified.

\begin{prop}\label{LTotAdjunction}
  Let $S\subseteq\{0,1,2,\ldots,r\}$ containing both $0$ and $r$.
  There is a Quillen adjunction $\inadj{\mc{L}}{\fChS}{\bChS}{\Totp}$.
\end{prop}
\begin{proof}
  This follows by definition of the weak equivalences and fibrations in $\bChS$, they are those morphisms taken to weak equivalences and fibrations in $\fChS$.
\end{proof}

\begin{theo}\label{LTotEquivalence}
  Let $S\subseteq\{0,1,2,\ldots,r\}$ containing both $0$ and $r$.
  There is a Quillen equivalence $\fChS\Quill\bChS$ induced by \cref{LTotAdjunction}.
\end{theo}
We will apply Greenlees and Shipley's cellularization principle, \cref{cellularizationPrinciple}, to the adjunction $\inadj{\mc{L}}{\fChS}{\bChS}{\Totp}$ of \cref{LTotAdjunction} where we right Bousfield localise at the $s$-cycles of $\fChS$ for $s\in S$.
We will show that these right Bousfield localisations don't change the model structures.
\begin{proof}
  In the setup of \cref{cellularizationPrinciple} we set $\mc{M}=\fChS$, $\mc{N}=\bChS$, $F=\mc{L}$ and $U=\Totp$ noting that both $\mc{M}$ and $\mc{N}$ are right proper (by construction) and cellular (\cref{filteredChainsAreCellular,bicomplexesAreCellular}) model categories. We take as our set of objects $\mc{K}=\{\Z_s(p,n)\}_{s\in S, p,n\in\Zbb}$ the set given by Dugger's theorem (\cref{duggersTheorem}) applied to $\fChS$ noting it is a left proper and cofibrantly generated model category.

  Part 1 of \cref{cellularizationPrinciple} is then immediate, i.e.\ there's a Quillen adjunction:
  \begin{equation}\label{CellularizedAdjunction}
    \adj{\mc{L}}{\mc{K}\mathdash\cellularization\mathdash\fChS}
    {\mc{L}\mc{K}\mathdash\cellularization\mathdash\bChS}{\Totp}
  \end{equation}
  and note we've dropped the cofibrant replacement $Q$ from the adjunction since all $\Z_s$ are cofibrant in $\fChS$.
  By \cref{duggersTheorem} the model structures $\fChS$ and $\mc{K}\mathdash\cellularization\mathdash\fChS$ agree by our choice of $\mc{K}$. We now identify the model structure $\mc{L}\mc{K}\mathdash\cellularization\mathdash\bChS$ with $\bChS$.

  We consider the set $\mc{L}\mc{K}$.
  By \cref{LOfSCycle} the set $\mc{L}\mc{K}$ consists of direct sums of witness $s$-cycles and witness $0$-cycles for $s\in S$.
  Bousfield localising at such a set is equivalent to Bousfield localising at the set of witness $s$-cycles since $0\in S$.
  However such a set is the same as obtained by applying \cref{duggersTheorem} to $\bChS$, and hence the right Bousfield localisation of $\bChS$ at $\mc{L}\mc{K}$ does not change the model structure.
  Hence Adjunction \ref{CellularizedAdjunction} becomes:
  \begin{equation*}
    \adj{\mc{L}}{\fChS}{\bChS}{\Totp}
  \end{equation*}
  and it remains to show the conditions of part 2 of \cref{cellularizationPrinciple}.

  The set $\mc{K}=\{\Z_s(p,n)\}_{s\in S, p,n\in\Zbb}$ is a stable set under $\rLoops$ and $\rSusp$ and each $s$-cycle is homotopically small in $\fChS$ by \cite[Theorem 7.4.3]{H}.
  Similarly elements of $FQ\mc{K}=\mc{L}\mc{K}$ are homotopically small by the same theorem: we can ignore the $0$-witness cycle summands of elements of $\mc{L}\mc{K}$ as they vanish in the homotopy category.
  Lastly by \cref{unitIsSIsomorphism} the derived unit of the adjunction is a weak equivalence on each element of $\mc{L}\mc{K}$.
  
  Hence by \cref{cellularizationPrinciple} the adjunction $\mc{L}\dashv\Totp$ is a Quillen equivalence.
\end{proof}
As a corollary we obtain all the model structures $\fChS$ ($0$ need not be in $S$ here) and $\bChS$, as $S$ varies, present the same homotopy category.

\begin{coro}\label{allModCatsQuillenEquiv}
  Let $S$ and $S'$ be finite subsets of $\mathbb{N}\cup\{0\}$ and $T$ and $T'$ be finite subsets of $\mathbb{N}\cup\{0\}$ containing $0$. Then there are zig-zags of Quillen equivalences between any two of $\fChS$, $\left(\fCh\right)_{S'}$, $\left(\bCh\right)_T$ and $\left(\bCh\right)_{T'}$.
  In particular their homotopy categories are all equivalent.
\end{coro}
\begin{proof}
  By the identity-identity and shift-d\'ecalage adjunctions any $\fChS$ is Quillen equivalent to $\left(\fCh\right)_{\{0\}}$. By \cref{LTotEquivalence} any $\left(\bCh\right)_T$ is Quillen equivalent to $\left(\fCh\right)_T$ via $\mc{L}\dashv\Totp$.
  It then follows any two of the model structures on $\fCh$ or $\bCh$ have a zig-zag of Quillen equivalences between them.
\end{proof}
\subsection{A distributive lattice of $S$-model structures on $\fCh$}\label{distributiveLatticeN}
We finish by describing a distributive lattice structure on the \textit{poset of model categories of the form $\fChS$}.
By this we mean the poset of finite subsets of $\mathbb{N}\cup \left\{0\right\}$ where $T\leq S$ if there is a composite of left adjoints of the identity-identity and shift-d\'ecalage adjunctions from the $T$-model structure to the $S$-model structure on $\fCh$.
We denote this poset by $\mc{N}$.
For those elements whose maxima are $3$ or less this is depicted in \cref{figu:posetN}.

Recall for a set $S$ we write $S\pm a\coloneqq \left\{s\pm a\,|\, s\in S\right\}$.
The above structure is equivalently generated by the two assertions that $T<S$ if either:
\begin{itemize}
\item $T\subset S$ and $\max T = \max S$ (describing the identity-identity Quillen equivalences of \cref{ididQEquiv}), or
\item $S=T+1$ (describing the shift-d\'ecalage Quillen equivalences of \cref{SDecQEquiv} with $l=1$).
\end{itemize}
\begin{figure}[h]
  \centering
  \begin{tikzpicture}[xscale=3,yscale=2]
    \node (0) at (0,0) {$\{0\}$};
    \node (1) at ($(0)+(1,0)$) {$\{1\}$};
    \node (01) at ($(1)+(0,1)$) {$\{0,1\}$};
    \node (2) at ($(1)+(1,0)$) {$\{2\}$};
    \node (12) at ($(01)+(1,0)$) {$\{1,2\}$};
    \node (02) at ($(2)+(-0.5,0.75)$) {$\{0,2\}$};
    \node (012) at ($(02)+(0,1)$) {$\{0,1,2\}$};
    \node (3) at ($(2)+(1,0)$) {$\{3\}$};
    \node (23) at ($(12)+(1,0)$) {$\{2,3\}$};
    \node (13) at ($(02)+(1,0)$) {$\{1,3\}$};
    \node (123) at ($(012)+(1,0)$) {$\{1,2,3\}$};
    \node (03) at ($(13)+(1,0)$) {$\{0,3\}$};
    \node (023) at ($(03)+(0,1)$) {$\{0,2,3\}$};
    \node (013) at (3,1.5) {$\{0,1,3\}$};
    \node (0123) at ($(013)+(0,1)$) {$\{0,1,2,3\}$};
    \node (continue) at (4,1.25) {$\ldots$};
    \draw[->] (013) -- (0123);
    \draw[->] (02) -- (012);
    \draw[->] (13) -- (123);
    \draw[->] (03) -- (023);
    \draw[->] (02) -- (13);
    \draw[->] (13) -- (013);
    \draw[->] (03) -- (013);
    \draw[->] (12) -- (012);
    \draw[->] (0) -- (1);
    \draw[->] (1) -- (2);
    \draw[->] (2) -- (3);
    \draw[->] (1) -- (01);
    \draw[->] (2) -- (02);
    \draw[->] (3) -- (13);
    \draw[->] (3) -- (03);
    \draw[->] (3) -- (23);
    \draw[->] (012) -- (123);
    \draw[->] (123) -- (0123);
    \draw[->] (023) -- (0123);
    \draw[->, very thick, white] (01) -- (12);
    \draw[->] (01) -- (12);
    \draw[->, very thick, white] (12) -- (23);
    \draw[->] (12) -- (23);
    \draw[->, very thick, white] (23) -- (123);
    \draw[->] (23) -- (123);
    \draw[->, very thick, white] (23) -- (023);
    \draw[->] (23) -- (023);
    \draw[->, very thick, white] (2) -- (12);
    \draw[->] (2) -- (12);
  \end{tikzpicture}
  \caption{The poset $\mc{N}$}
  \label{figu:posetN}
\end{figure}
Recall a \textit{lattice} is a poset $\Lambda$ equipped with binary operations $\vee$, called join, and $\wedge$, called meet, on its elements such that $a\leq a\vee b=b\vee a$ and $a\wedge b=b\wedge a\leq a$.
It is further a \textit{distributive lattice} if for all $a,b,c\in \Lambda$ we have
$a\vee(b\wedge c)=(a\vee b)\wedge(a\vee c)$ and $a\wedge(b\vee c)=(a\wedge b)\vee (a\wedge c)$.

We can equip the poset $\mc{N}$ with binary operations which for elements $S$ and $T$ can informally be described as ``the initial model category admitting left adjoints from $S$ and $T$'' and ``the terminal model category admitting left adjoints to $S$ and $T$'' respectively for those model categories indexed by $\mc{N}$.
Note we are only considering the identity-identity and shift-d\'ecalage Quillen equivalences here, we are not precluding existence of other Quillen adjunctions between the model structures on $\fCh$ indexed by elements of $\mc{N}$.

We denote these operations by $\vee$ and $\wedge$ respectively and in set notation can be described as:
\begin{itemize}\label{MeetJoinFormulae}
\item $S\vee T = \left(S+\max\left\{S\cup T\right\}-\max S\right)\cup\left(T+\max\left\{S\cup T\right\}-\max T\right)$, and
\item $S\wedge T = \left(S-\max\left\{S\cup T\right\}+\max T\right)\cap\left(T-\max\left\{S\cup T\right\}+\max S\right)$.
\end{itemize}
We show these describe a distributive lattice structure on $\mc{N}$.
We do so indirectly by showing it is isomorphic to another distributive lattice structure.
\begin{defi}
  For a distributive lattice $\Lambda$ an element $a\in\Lambda$ is said to be \textit{join-irreducible} if it is neither the least element of the lattice nor the join of two smaller elements.
\end{defi}
\begin{defi}
  A \textit{lower set} of $L$ of a lattice $\Lambda$ is a subset $L\subseteq\Lambda$ such that if $l\in L, \lambda\in\Lambda$ with $\lambda\leq l$ then $\lambda\in L$.
\end{defi}
\begin{theo}[Birkhoff's representation theorem]\label{Birkhoff}
  Any finite distributive lattice $\Lambda$ is isomorphic to the distributive lattice on the set of lower sets of the partial order on the join-irreducible elements with meet and join the operations the usual set theoretic intersection and union.\qed
\end{theo}
\begin{proof}
  \cite[Theorem 5.12]{DaveyPriestley}
  The isomorphism is given by sending a $\lambda\in\Lambda$ to the set of join-irreducible elements less than or equal to $\lambda$, and conversely sends a lower set of join-irreducibles to the join of its elements.
\end{proof}

Write $\mc{N}_r$ for the sub-poset of $\mc{N}$ consisting of those elements whose maximum is $r$ or less.
Note the operations $\vee$ and $\wedge$ restrict to $\mc{N}_r$.
We can still define join-irreducible elements in a lattice instead of a distributive lattice.
\begin{lemm}
  The join-irreducible elements of $\mc{N}$ are those elements of the form $\left\{n\right\}$ or $\left\{0,n\right\}$ for $n\geq 1$ and for join-irreducibles of $\mc{N}_r$ we also require $n\leq r$.\qed
\end{lemm}
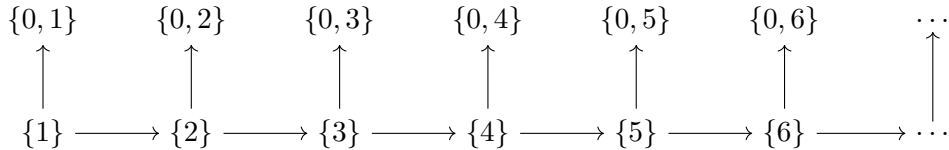
\begin{figure}[h]
  \centering
  \begin{tikzpicture}[xscale=1.95,yscale=1.5]
    \newcounter{TempVar}
    \newcounter{PosetLength}
    \setcounter{PosetLength}{6}
    \newcounter{PosetLengthPOne}
    \setcounter{PosetLengthPOne}{\thePosetLength+1}
    \newcounter{PosetLengthMOne}
    \setcounter{PosetLengthMOne}{\thePosetLength-1}
    \node (inf 1) at (\thePosetLengthPOne,1) {$\ldots$};
    \node (inf 2) at (\thePosetLengthPOne,2) {$\ldots$};
    \foreach \x in {1,...,\thePosetLength}{
      \node (\x 1) at (\x,1) {$\{\x\}$};
      \node (\x 2) at (\x,2) {$\{0,\x\}$};
    }
    \foreach \x in {1,...,\thePosetLengthMOne}{
      \setcounter{TempVar}{\x+1}
      \draw[->] (\x 1) -- (\theTempVar 1);
    }
    \foreach \x in {1,...,\thePosetLength}{
      \draw[->] (\x 1) -- (\x 2);
    }
    \draw[->] (\thePosetLength 1) -- (inf 1);
    \draw[->] (inf 1) -- (inf 2);
  \end{tikzpicture}\caption{The poset of join-irreducibles of $\mc{N}$}
  \label{figu:JIN}
\end{figure}
\begin{lemm}\label{lowerSetsOfNr}
  The lower sets of $\mc{N}_r$ are of the form $\emptyset$ or:
  \begin{equation}
    \left\{\left\{1\right\},\left\{2\right\},\ldots,\left\{s\right\}\right\}\cup
    \left\{\left\{0,t_1,\right\},\left\{0,t_2\right\},\ldots,\left\{0,t_k\right\}\right\}
  \end{equation}
  where $s\leq r$, the first set contains all elements $1\leq i\leq s$ and $\max_i t_i\leq s$.\qed
\end{lemm}
\begin{lemm}\label{posetIsoNrAndLowerSets}
  There is an isomorphism of posets between $\mc{N}_r$ and the lower sets of join-irreducible elements of $\mc{N}_r$.
\end{lemm}
\begin{proof}
  A lower set of the form given in \cref{lowerSetsOfNr} is sent to the join of its elements in $\mc{N}_r$, the set:
  \begin{equation*}
    \left\{s-t_1,s-t_2,\ldots,s-t_k,s\right\}
  \end{equation*}
  and $\emptyset$ is sent to $\left\{0\right\}$.
  Conversely an element of $S=\left\{t_1,t_2,\ldots, t_k,s\right\}$ of $\mc{N}_r$ is sent to the set of join-irreducibles in $\mc{N}_r$ less than it:
  \begin{equation*}
    \left\{\left\{1\right\},\left\{2\right\},\ldots,\left\{s\right\}\right\}\cup
    \left\{\left\{0,s-t_1\right\},\left\{0,s-t_2\right\},\ldots,
      \left\{0,s-t_k\right\}\right\}
  \end{equation*}
  and $\left\{0\right\}$ is sent to $\emptyset$.
  These can be checked to be order preserving and inverse to each other.
\end{proof}

Write $\alpha$ for this construction sending an element of $\mc{N}_r$ to the lower set and $\beta$ for its inverse.

\begin{lemm}
  We have that:
  \begin{align*}
    \alpha(A\vee B)&=\alpha(A)\cup \alpha(B)\,,
    &\alpha(A\wedge B)&=\alpha(A)\cap\alpha(B)\,.
  \end{align*}
  In particular $\alpha$ and $\beta$ preserve the join and meet operations.\qed
\end{lemm}

\begin{coro}
  The lattice structures on $\mc{N}_r$ and $\mc{N}$ are distributive lattices.
\end{coro}
\begin{proof}
  The lattice structure on $\mc{N}_r$ was shown to be isomorphic to that on the lattice of lower sets of join-irreducible elements in \cref{posetIsoNrAndLowerSets} which is a distributive lattice, hence so too is $\mc{N}_r$.

  The case for $\mc{N}$ now follows by verifying the distributive relations in one of the sub-lattices $\mc{N}_r$ since the elements of $\mc{N}$ are finite subsets.
\end{proof}
